\documentclass[11pt]{article}

\usepackage{
	amsbsy,
	amsfonts,
	amsmath,
	amssymb,
	amsthm,
	array,
    authblk,
	blindtext,
	booktabs,
	centernot,
	changepage,
	enumerate,
	fancyhdr,
	float,
	graphicx,
	hhline,
	hyperref,
	leftindex,
	letltxmacro,
	mathptmx,
	mathtools,
	multicol,
	multirow,
	qtree,
	scalerel,
	subfiles,
    subcaption,
	tabto,
	tikz,
	tikz-cd,
	titlesec,
	xcolor,
	xfrac,
	xparse}

\DeclareMathAlphabet{\matholdcal}{OMS}{cmsy}{m}{n}
\usepackage{dutchcal}

\usepackage[ruled,vlined]{algorithm2e}
\SetKw{Ensure}{Ensure}

\usetikzlibrary {patterns,patterns.meta}
\usepackage[margin=1in]{geometry}
\usepackage[
backend=biber,
style=alphabetic,
sorting=nty
]{biblatex}
\hypersetup{
    colorlinks=true,
    linkcolor=blue,
    filecolor=magenta,      
    urlcolor=cyan,
    citecolor=magenta,
    pdfpagemode=FullScreen,
    }
\usetikzlibrary{decorations.pathmorphing}
\newcounter{def}

\newtheoremstyle{myDefinition}{}{}{}{}{\color{blue} \bfseries}{:}{ }{}
\newtheoremstyle{myTheorem}{}{}{}{}{\color{violet} \bfseries}{:}{ }{}
\newtheoremstyle{myLemma}{}{}{}{}{\color{orange} \bfseries}{:}{ }{}
\newtheoremstyle{myCorollary}{}{}{}{}{\color{magenta} \bfseries}{:}{ }{}
\newtheoremstyle{myProposition}{}{}{}{}{\color{brown} \bfseries}{:}{ }{}
\newtheoremstyle{myRemark}{}{}{}{}{\color{red} \bfseries}{:}{ }{}
\newtheoremstyle{myExample}{}{}{}{}{\color{teal} \bfseries}{:}{ }{}
\newtheoremstyle{myConjecture}{}{}{}{}{\color{red} \bfseries}{:}{ }{}
\newtheoremstyle{mySolution}{}{}{}{}{\color{cyan}}{:}{ }{}

\theoremstyle{myDefinition}
\newtheorem{definition}[def]{Definition}
\theoremstyle{myTheorem}
\newtheorem{theorem}[def]{Theorem}
\theoremstyle{myLemma}
\newtheorem{lemma}[def]{Lemma}
\theoremstyle{myCorollary}
\newtheorem{corollary}[def]{Corollary}
\theoremstyle{myProposition}
\newtheorem{proposition}[def]{Proposition}
\theoremstyle{myExample}

\theoremstyle{mySolution}

\theoremstyle{myRemark}
\newtheorem{remark}[def]{Remark}
\theoremstyle{myConjecture}

\renewenvironment{proof}
    {\LetLtxMacro\itemold\item
    \begin{enumerate}
        \item[\color{cyan}\underline{Proof:}\color{black}]\renewcommand{\item}{\itemindent0cm\itemold}}
    {\hfill $\blacksquare$ \end{enumerate}}

\newcommand{\F}{\mathbb{F}}

\newcommand{\Q}{\mathbb{Q}}
\newcommand{\Z}{\mathbb{Z}}

\newcommand{\Ccal}{\mathcal{C}}

\newcommand{\Ecal}{\matholdcal{E}}
\newcommand{\Fcal}{\matholdcal{F}}
\newcommand{\Gcal}{\mathcal{G}}

\newcommand{\Ocal}{\mathcal{O}}

\newcommand{\Tcal}{\matholdcal{T}}

\newcommand{\pfrak}{\mathfrak{p}}

\let\oldexists\exists
\renewcommand{\exists}{\oldexists\;}
\let\oldforall\forall
\renewcommand{\forall}{\oldforall\;}
\newcolumntype{M}[1]{>{\centering\arraybackslash}m{#1}}

\newcommand{\End}{\texttt{End}}

\newcommand{\ord}{\texttt{ord}}

\NewDocumentCommand{\GL}{ O{n} O{\mathbb{C}}}{GL_{#1}(#2)}
\NewDocumentCommand{\SL}{ O{n} O{\mathbb{C}}}{SL_{#1}(#2)}
\NewDocumentCommand{\Mat}{ O{n} O{\mathbb{C}}}{M_{#1}(#2)}

\renewcommand{\pmod}[1]{\mbox{\ $(\mathrm{mod}\ {#1})$}}

\renewcommand{\deg}{\texttt{deg}}
\newcommand{\diameter}{\texttt{diam}}
\newcommand{\periphery}{\texttt{per}}
\newcommand{\ecc}{\texttt{ecc}}
\newcommand{\MeanDiam}{\texttt{MeanDiam}}

\NewDocumentCommand{\slp}{ O{\ell} O{p} }{\mathcal{S}_{#1}^{#2}}
\NewDocumentCommand{\glfb}{ O{\ell} O{p} }{\mathcal{G}_{#1}(\overline{\mathbb{F}_{#2}})}
\NewDocumentCommand{\glf}{ O{\ell} O{p} }{\mathcal{G}_{#1}(\mathbb{F}_{#2})}

\newcommand{\normneq}{\mathrel{\ooalign{$\lneq$\cr\raise.22ex\hbox{$\lhd$}\cr}}}

\newcommand{\tol}[1]{\color{teal} #1\color{black}\! }

\newcommand{\bol}[1]{\color{blue} #1\color{black}\! }

\newcommand{\rol}[1]{\color{red} #1\color{black}\! }

\DeclarePairedDelimiter{\abs}{\lvert}{\rvert}

\newcommand{\tabs}[1]{{\left\vert\kern-0.25ex\left\vert\kern-0.25ex\left\vert #1 \right\vert\kern-0.25ex\right\vert\kern-0.25ex\right\vert}}
\DeclarePairedDelimiter{\set}{\{}{\}}
\DeclarePairedDelimiter{\gen}{\langle}{\rangle}
\DeclarePairedDelimiter{\brac}{(}{)}
\DeclarePairedDelimiter{\sbrac}{[}{]}
\DeclarePairedDelimiter{\ceil}{\lceil}{\rceil}
\DeclarePairedDelimiter{\floor}{\lfloor}{\rfloor}
\allowdisplaybreaks

\newcommand{\calmImage}[1]{
\begin{center}
\scalebox{#1}{
\begin{tikzpicture}

\draw[thick] (-2,0)--(0,2)--(2,0)--(0,-2)--(-2,0);
\draw[thick] (-2,0)--(1,1)--(0,-2)--(-1,1)--(2,0)--(-1,-1)--(0,2)--(1,-1)--(-2,0);
\draw[thick] (0,2)--(0,-2);
\draw[thick] (-2,0)--(2,0);
\draw[thick](-1,1)--(1,-1);
\draw[thick](1,1)--(-1,-1);

\draw[thick](-5,5)--(-5,-5);
\draw[thick](-5,5)--(-6,-5);
\draw[thick](-5,5)--(-7,-5);
\draw[thick](-5,5)--(-8,-5);
\draw[thick](-5,5)--(-9,-5);

\draw[thick](-6,5)--(-5,-5);
\draw[thick](-6,5)--(-6,-5);
\draw[thick](-6,5)--(-7,-5);
\draw[thick](-6,5)--(-8,-5);
\draw[thick](-6,5)--(-9,-5);

\draw[thick](-7,5)--(-5,-5);
\draw[thick](-7,5)--(-6,-5);
\draw[thick](-7,5)--(-7,-5);
\draw[thick](-7,5)--(-8,-5);
\draw[thick](-7,5)--(-9,-5);

\draw[thick](-8,5)--(-5,-5);
\draw[thick](-8,5)--(-6,-5);
\draw[thick](-8,5)--(-7,-5);
\draw[thick](-8,5)--(-8,-5);
\draw[thick](-8,5)--(-9,-5);

\draw[thick](-9,5)--(-5,-5);
\draw[thick](-9,5)--(-6,-5);
\draw[thick](-9,5)--(-7,-5);
\draw[thick](-9,5)--(-8,-5);
\draw[thick](-9,5)--(-9,-5);

\draw[thick](5,5)--(5,-5);
\draw[thick](5,5)--(6,-5);
\draw[thick](5,5)--(7,-5);
\draw[thick](5,5)--(8,-5);
\draw[thick](5,5)--(9,-5);

\draw[thick](6,5)--(5,-5);
\draw[thick](6,5)--(6,-5);
\draw[thick](6,5)--(7,-5);
\draw[thick](6,5)--(8,-5);
\draw[thick](6,5)--(9,-5);

\draw[thick](7,5)--(5,-5);
\draw[thick](7,5)--(6,-5);
\draw[thick](7,5)--(7,-5);
\draw[thick](7,5)--(8,-5);
\draw[thick](7,5)--(9,-5);

\draw[thick](8,5)--(5,-5);
\draw[thick](8,5)--(6,-5);
\draw[thick](8,5)--(7,-5);
\draw[thick](8,5)--(8,-5);
\draw[thick](8,5)--(9,-5);

\draw[thick](9,5)--(5,-5);
\draw[thick](9,5)--(6,-5);
\draw[thick](9,5)--(7,-5);
\draw[thick](9,5)--(8,-5);
\draw[thick](9,5)--(9,-5);

\draw[thick](0,8)--(-5,5);
\draw[thick](0,8)--(-6,5);
\draw[thick](0,8)--(-7,5);
\draw[thick](0,8)--(-8,5);
\draw[thick](0,8)--(-9,5);

\draw[thick](0,8)--(5,5);
\draw[thick](0,8)--(6,5);
\draw[thick](0,8)--(7,5);
\draw[thick](0,8)--(8,5);
\draw[thick](0,8)--(9,5);

\draw[thick](0,8)--(-2,0);
\draw[thick](0,8)--(-1,1);
\draw[thick](0,8)--(0,2);
\draw[thick](0,8)--(1,1);
\draw[thick](0,8)--(2,0);

\draw[thick](0,8)--(-1,-1);
\draw[thick](0,8)--(1,-1);

\draw[thick](0,-8)--(-1,1);
\draw[thick](0,-8)--(1,1);

\draw[thick](0,-8)--(-5,-5);
\draw[thick](0,-8)--(-6,-5);
\draw[thick](0,-8)--(-7,-5);
\draw[thick](0,-8)--(-8,-5);
\draw[thick](0,-8)--(-9,-5);

\draw[thick](0,-8)--(5,-5);
\draw[thick](0,-8)--(6,-5);
\draw[thick](0,-8)--(7,-5);
\draw[thick](0,-8)--(8,-5);
\draw[thick](0,-8)--(9,-5);

\draw[thick](0,-8)--(-2,0);
\draw[thick](0,-8)--(-1,-1);
\draw[thick](0,-8)--(0,-2);
\draw[thick](0,-8)--(1,-1);
\draw[thick](0,-8)--(2,0);

\draw[thick](-2,0)--(-5,5);
\draw[thick](-2,0)--(-6,5);
\draw[thick](-2,0)--(-7,5);
\draw[thick](-2,0)--(-8,5);
\draw[thick](-2,0)--(-9,5);

\draw[thick](2,0)--(5,5);
\draw[thick](2,0)--(6,5);
\draw[thick](2,0)--(7,5);
\draw[thick](2,0)--(8,5);
\draw[thick](2,0)--(9,5);

\draw[thick](-2,0)--(-5,-5);
\draw[thick](-2,0)--(-6,-5);
\draw[thick](-2,0)--(-7,-5);
\draw[thick](-2,0)--(-8,-5);
\draw[thick](-2,0)--(-9,-5);

\draw[thick](2,0)--(5,-5);
\draw[thick](2,0)--(6,-5);
\draw[thick](2,0)--(7,-5);
\draw[thick](2,0)--(8,-5);
\draw[thick](2,0)--(9,-5);

\end{tikzpicture}}
\end{center}}

\title{The Spine: A Supersingular Highway}

\author{Taha Hedayat}
\affil{University of Calgary}

\date{April, 14, 2026}
\begin{document}

\maketitle
\begin{abstract}
    We consider the structure of the spine of the supersingular $\ell$-isogeny graph for one of the cases which \cite{HedayatArpinScheidler} was not able to fully describe, $\ell = 2$ and $p = 71, 119\pmod{120}$. We find the distance, eccentricity, and diameter functions, of the components of the spine without the non-trivial edge not defined over $\F_p$. Using these functions, we find the mean diameter of the spine and show how this value distinguishes the different structures of the spine. Thus, allowing us to use explicit computations to provide heuristics on the behavior of the spine's structure as $p$ varies.
\end{abstract}
\calmImage{0.4}
\textit{Supervised by:}
\begin{itemize}
\item Dr. Renate Scheidler - University of Calgary

\item Dr. Sarah Arpin - Virginia Polytechnic Institute and State University
\end{itemize}

\textit{Acknowledgments:} Great thanks to Dr. Renate Scheidler and Dr. Sarah Arpin for their support, education, guidance, and work. Without both of their help none of the work in this paper could be possible. We give great thanks to Dr. Tracey Balehowsky for their invaluable help in navigating the writing of the thesis. Lastly, we would like to thank Rosa Shah for their help in finding explicit formulation of some particularly difficult functions.

\newpage

\section{Introduction}

Through recent works in post-quantum cryptography, the research field of \emph{isogeny-based cryptography} emerged in the early 2000s. It has since produced key exchange protocols \cite{SCALLOP}, hash functions \cite{HashFunction}, and signature schemes \cite{SQIsign}. The cryptographic problem of all of these is computationally equivalent to the \emph{path finding problem}. The path finding problem states the following: 
\begin{quote}
    Given a small prime $\ell$, a cryptographically large prime $p$, and two \emph{supersingular} elliptic curves over $\overline{\F_p}$, provide a non-zero group homomorphism (called an \emph{isogeny}) of the form $\phi:E_1 \rightarrow E_2$.
\end{quote}

While this problem seems easy at first, it is anything but. In particular, there is no known polynomial-time classical nor quantum solution for the path finding problem. Currently, the best classical solution for the general path finding problem requires $\tilde{O}(\sqrt{p})$ bit operations. The heuristics on the difficulty of the path finding problem comes from the Deuring Correspondence which connects elliptic curves, their isogenies, and their endomorphism rings to maximal orders within the quaternion ramified at $p$ and infinity. As a key consequence, certain elliptic curves have endomorphism rings which are easily computed. In particular, the elliptic curves with $j$-invariant lying in $\F_p$ have endomorphism rings isomorphic to either $\Z[\sqrt{-p}]$ or $\Z\sbrac*{\frac{1 + \sqrt{-p}}{2}}$. Consequently \cite{DelfsGalbraith} showed that one could use the Deuring Correspondence for these curves to find a chain of isogenies (hence making an isogeny by composing them) with a much faster algorithm, having complexity $\tilde{O}(\sqrt[4]{p})$. This is a significant drop in the number of operations required since $p$ is cryptographically large. Crucially, \cite{DelfsGalbraith} also provided descriptive theorems on the structure of the graph generated when considering $\F_p$-isomorphism classes of elliptic curves defined over $\F_p$ and $\F_p$-equivalence classes of isogenies defined over $\F_p$. This graph is denoted by $\glf$. Motivated by the discoveries of \cite{DelfsGalbraith}, the authors of \cite{Adventures} dived deep into understanding how the graph theoretic structure of $\glf$ changes when instead of considering the graph defined over $\F_p$, we transitioned into instead consider isomorphism classes up to $\overline{\F_p}$. This new graph is denoted $\slp$, and it is in this paper that the term ``spine" was coined to represent the subgraph of $\glfb$ which has $j$-invariants defined over $\F_p$. \cite{Adventures} set the stage for the work done in \cite{HedayatArpinScheidler}. It refined the findings in \cite{Adventures} and \cite{DelfsGalbraith} to provide explicit graph-theoretic descriptions of $\slp$ and $\glf$ based on congruence conditions for $p$ and the value of $\ell$. While the descriptions in \cite{HedayatArpinScheidler} is the most explicit description so far, the congruence conditions on $p$ do not always provide a deterministic structure on $\slp$. Specifically, no deterministic criterion is known for determining when a new edge connects components.

For example, when $p \equiv 71, 119\pmod{120}$ and $\ell = 2$, the structure of $\glf$ is known to be a union of volcano-style components. It is known that when we go from $\glf$ to $\slp$, most of the components stack, except for one, which folds onto its self. It is also known that a new edge is added when we land in $\slp$. However, it is not known precisely for which $p$ the new edge connects two initially disconnected components and when it does not. It is important to know when the new edge connects two disconnected components of the graph, due to its increased number of vertices which are connected together via paths. As such resulting in the exploit found by \cite{DelfsGalbraith} to have greater affect. This thesis is dedicated to understanding the structure of $\slp$ when $\ell = 2$ and $p \equiv 71, 119 \pmod{120}$. We hope that exploring the structure of $\slp$ in this case, will bring light to the structure of $\slp$ in other cases. Our main contribution to the work of \cite{HedayatArpinScheidler} is a methodology for differentiating the different potential structures of $\slp$ and the heuristics we provide on how often each potential structure described in \cite{HedayatArpinScheidler} occur.

We begin in Section \ref{sec:aBitOfBackground} by providing definitions and providing results which are mostly known by the community of researchers in this field. In Section \ref{sec:UnderstandingGammaGlf} we look closer at the structure of $\Gamma(\glf)$ for when $p \equiv 71, 119\pmod{120}$ and $\ell = 2$ and determine the diameters of its components. In Section \ref{sec:understandingThetaGammaGlf} we explore all possible structures that may arise from the new edge. Using this information we, in Section \ref{sec:understandingMeanDiameter}, provide a distinction in the value of the mean diameter in each one of the possible structures of the spine. Thus, allowing us to compute the explicit value of the mean diameter of the spine and determine what the actual structure of the spine tends to be, which we do in Section \ref{sec:exploringDiameterData}.

\section{Background}\label{sec:aBitOfBackground}

This paper will touch on topics in number theory as well as graph theory. If the reader requires a starting place for any graph theoretic ideas we recommend consulting the first section of \cite{harris} as it follows much of the notation that we will use here. Similarly, if the reader requires a starting place for any algebraic number theory ideas we recommend reading \cite{marcus} as it is a rich resource.

We give a quick reminder to the reader of the graph theoretic values/ideas used in this paper. Note that, these definitions defer depending on the type of graph used, but the following holds valid for our purposes:
\begin{definition}
    Let $G$ be a graph with $V := V(G)$ as its set of vertices and $E:=E(G)$ as its set of edges. 
    \begin{enumerate}
        \item A \bol{walk} in $G$ is an ordered sequence of elements in $V$ such that consequent elements in the path form an edge in $G$. If no vertex is repeated within this walk, then it is called a \bol{path}.
        \item The \bol{length} of a path $P$, denoted \bol{$\abs{P}$}, is the number of vertices in the path minus 1.
        \item The \bol{distance} from $v\in V$ to $w\in W$, denoted \bol{$d_G(v,w)$} or \bol{$d(v,w)$} when context is clear, is the minimum length of a path which starts at $v$ and ends at $w$. If no path exists from $v$ to $w$ then the distance value is $\infty$ by convention.
        \item The \bol{eccentricity} of a vertex $v \in V$, denoted \bol{$\ecc_G(v)$} or \bol{$\ecc(v)$}, is the maximum distance from $v$ to any other vertex.
        \item The \bol{diameter} of $G$, denoted \bol{$\diameter(G)$}, is the maximum eccentricity of the vertices of $G$.
        \item The \bol{periphery} of $G$, denoted \bol{$\periphery(G)$}, is the subset of $V$ made from vertices which have eccentricity equal to the diameter of $G$.
        \item $G$ is said to be a \bol{symmetric and strongly connected} graph if for every $u,v \in V$, $d(u,v) = d(v,u) < \infty$.
        \item The \bol{order of $G$}, denoted $\bol{\abs{V(G)}}$ or $\bol{\abs{G}}$, is the cardinality of the vertex set of $G$, $V$.
        \item For $v \in V$, the \bol{neighbourhood of $v$}, denoted $\bol{N_G(v)}$ or $\bol{N(v)}$, is the subset of $V$ made from vertices which are adjacent to $v$.
        \item For $v \in V$ the \bol{degree of $v$}, denoted $\bol{\deg_G(v)}$ or $\bol{\deg(v)}$, is the cardinality of $N_G(v)$.
    \end{enumerate}
\end{definition}

\begin{remark}
    There is an unfortunate notational issue with the order of a graph and the length of a path. To avoid this, we use $\abs{P}$ for the length of a path, and $\abs{V(P)}$ for the order of a path.
\end{remark}

Some non-standard notation is required to reference concepts that are not commonly discussed. We require a measure for the diameter of all components of a disconnected graph simultaneously. Motivated by this we define the following:

\begin{definition}
    Let $G$ be a graph whose components are strongly connected. Let $\set{C_1, \ldots, C_m}$ be the connected components of $G$. We define the \bol{mean diameter} of $G$ to be
    $$\MeanDiam(G) = \frac{\displaystyle \sum_{i=1}^m \diameter(C_i)}{m}$$
\end{definition}

In addition we will need to be able to talk about the order of elements in the class group. Hence we have the following definition/notation.
\begin{definition}
    Let $\ell=2$ and let $p > 3$ be a prime. Let $\pfrak$ be a prime above $\ell$ in $\Q[\sqrt{-p}]$. The \bol{order of 2} (or the \bol{order of $\ell$}) is the order of $\pfrak$ in the class group, and is denoted \bol{$\ord(2)$} or \bol{$\ord(\ell)$}.
\end{definition}

By Lagrange's Theorem, $\ord(\ell)$ must divide the class number $h(-p)$. 

The definition of an elliptic curve defined over a field $k$, denoted $E/k$, and its isomorphism classes over a superfield $L$ are well-known. However, equivalence of isogenies prove elusive. We provide the full definition of equivalent isogenies for clarity:

\begin{definition}\label{def:equivIsogeny}
    Let $E_1/k$, $E_2/k$, $E_1'/k$, $E_2'/k$ be elliptic curves such that $E_1 \neq E_2$. Let $\phi_1 : E_1 \rightarrow E_1'$ and $\phi_2: E_2 \rightarrow E_2'$ be separable isogenies, and let $L$ be a superfield of $k$. We say that $\phi_1$ and $\phi_2$ are \bol{equivalent over the field $L$}, denoted $\phi_1 \bol{\sim_L} \phi_2$, if there exists $L$-isomorphisms $\lambda:E_1 \rightarrow E_2$, and $\lambda' : E_1' \rightarrow E_2'$ such that $\phi_2 \circ \lambda = \lambda'\circ \phi_1$. In other words, $\phi_1$ and $\phi_2$ are equivalent over $L$ if there exists $L$-isomorphism $\lambda$ and $\lambda'$ such that the following diagram commutes:
    \[\begin{tikzcd}
	{E_1} && {E_2} \\
	\\
	{E_1'} && {E_2'}
	\arrow["\lambda", from=1-1, to=1-3]
	\arrow["{\phi_1}"', from=1-1, to=3-1]
	\arrow["{\phi_2}", from=1-3, to=3-3]
	\arrow["{\lambda'}"', from=3-1, to=3-3]
    \end{tikzcd}\]
\end{definition}

\begin{remark}
    If $E_1 = E_2$ in the above definition, then equivalence of curves is only counted up to post-composition: Let $E/k,\ E_1'/k,\ E_2'/k$ be elliptic curves. Let $\phi_1:E \rightarrow E_1'$ and $\phi_2:E \rightarrow E_2'$ be separable isogenies, and let $L$ be a superfield of $k$. We say that $\phi_1$ and $\phi_2$ are equivalent over the field $L$ if there exists an $L$ isomorphism $\lambda':E_1'\rightarrow E_2'$ such that $\lambda' \circ \phi_1 = \phi_2$. Another interpretation is that two isogenies are equivalent if and only if their kernels are equal to each other.
\end{remark}

Hence we are now able to understand the definition of the Supersingular $\ell$-Isogeny Graph over a field $k$:

\begin{definition}\label{def:SSIG}
    The \bol{supersingular $\ell$-isogeny graph over the field $k$} is the graph defined as follows:
    \begin{itemize}
        \item The vertices are $k$-isomorphism classes of supersingular elliptic curves defined over $k$,
        \item The edges are $k$-equivalence classes of isogenies between the $k$-isomorphism classes of supersingular elliptic curves defined over $k$, such that the isogeny is defined over $k$ and has a degree of $\ell$.
    \end{itemize}
    We denote this graph as $\bol{\Gcal_\ell(k)}$.
\end{definition}

As previously stated, our focus is on a subgraph of $\glfb$, namely the spine $\slp$.

\begin{definition}
     The \bol{spine of the supersingular $\ell$-isogeny graph over the field $\overline{\F_p}$} is the subgraph of $\glfb$ made of vertices with $j$-invariant in $\F_p$. This subgraph is denoted \bol{$\slp$}.
\end{definition}

The way that we will understand the spine is the same way that \cite{HedayatArpinScheidler} and \cite{Adventures} did. We start with $\glf$. In contrast to $\glfb$, $\glf$ has very explicitly known structure as described initially by \cite[\nopp Theorem 2.7]{DelfsGalbraith}. \cite{Adventures} was the first to note that we can use $\glf$ to arrive at $\slp$ by first considering isomorphisms defined over $\overline{\F_p}$ rather than $\F_p$, and then adding any edges defined only over $\overline{\F_p}$. \cite{Adventures} worked on providing a description of this transition process. They showed how when one is changing from $\F_p$ isomorphism classes to $\overline{\F_p}$ isomorphism classes, $\F_p$ isomorphism classes of elliptic curves merge with their quadratic twists (except for the isomorphism class with $j$-invariant 1728 which merges with its quartic twist), and many of the edges merge due to the newly available isomorphisms. In particular, since each elliptic curve isomorphism class has only one supersingular twist, the merging of vertices happens in a 2-to-1 fashion. Furthermore, they showed that all edges (not going out of 1728) merge in a 2-to-1 fashion. We make this more precise in Proposition \ref{prop:1728LoopsMerge} and Proposition \ref{prop:isogeniesGoingOutOf1728NotLoopDontFold}. 

\cite{Adventures} further showed that components in $\glf$ are isomorphic to each other. Precicely stating, if there exists a component with all distinct $j$-invariants, then there exists a distinct component which is isomorphic to the original without changing the $j$-invariant labels (i.e. there is a component which is isomorphic by taking the quadratic twist of all vertices). On the other hand, if a component consists of a copy of a $j$-invariant, then the twist of every vertex is contained within the same component.

This yielded an explicit description of the spine, which we restate with the following definition:
\begin{definition}
    When changing from $\F_p$ isomorphisms to $\overline{\F_p}$ isomorphisms the following events can occur\footnote{In addition ``attachment by vertex'' is possible, but this is not relevant to the content of this paper.}:
    \begin{enumerate}
        \item \bol{Stacking}: The merging of two distinct components of $\glf$ that are isomorphic up to a quadratic twist of their constituent elliptic curves.
        \item \bol{Folding}: The self-merging of a component that contains both a vertex and its twist.
    \end{enumerate}
    When adding the edges not defined over $\F_p$ to the graph the following may happen:
    \begin{enumerate}
        \item[3.] \bol{Edge Attachment}: The addition of a new edge that connects two previously disconnected components of the graph.
    \end{enumerate}
\end{definition}
\vspace{-10pt}
\cite{HedayatArpinScheidler} formalized this transition using the functions $\Gamma$ and $\Theta$:
\begin{definition}\label{def:gftospfunctions}
    Let $\bol{\Gamma}: \glf \rightarrow \glfb$ be the function such that $\Gamma([E]_{\F_p}) = [E]_{\overline{\F_p}}$ and $\Gamma([\phi]_{\F_p}) = [\phi]_{\overline{\F_p}}$. Furthermore, define $\bol{\Theta}: \Gamma(\glf) \rightarrow \glfb$ such that if two vertices in $\Gamma(\glf) $ have isogenies of degree $\ell$ between them that are not defined over $\F_p$, then $\Theta$ adds those edges into the graph $\Gamma(\glf)$.
\end{definition}

By \cite[\nopp Theorem 4.8]{HedayatArpinScheidler}, when $\ell =2$ and $p \equiv 71, 119 \pmod{120}$ a single component folds (from now on called \bol{the folded component}), no vertex attachment takes place, and all other components stack (from now on called \bol{stacked components}). In particular one end of the folded component has the vertex with $j$-invariant 1728, and on the opposite end it has $j$-invariant 8000. By \cite{DelfsGalbraith}, the vertices adjacent to only one other vertex (except for 1728) have $\F_p$ endomorphism ring $\Z[\sqrt{-p}]$, while the others (except 1728) have $\F_p$ endomorphism ring $\Z\sbrac*{\frac{1 + \sqrt{-p}}{2}}$. This is shown in Figure \ref{fig:Gamma(G_2(F_p))withp=71_119_mod120}. 

For the remainder of this paper, assume $\ell = 2$ and $p \equiv 71, 119 \pmod{120}$, unless stated otherwise. With this in mind we can define specific terminology and provide some specific results which will be pertinent to the other sections of this paper.

\begin{definition}
    Let us denote the \bol{folded component} (the component made from folding) of $\Gamma(\glf)$ by \bol{$\Fcal$} and denote a \bol{stacked component} (a component made from stacking) by \bol{$\Tcal$}. Furthermore, for any vertex $v \in V(\Gamma(\glf))$, let $v$ be a \bol{floor vertex} if it has an $\F_p$ endomorphism ring $\Z[\sqrt{-p}]$ (excluding 1728), and let $v$ be a \bol{surface vertex} if it has an $\F_p$ endomorphism ring $\Z\sbrac*{\frac{1 + \sqrt{-p}}{2}}$ (including 1728).
\end{definition}

By \cite[\nopp Eq 1]{DelfsGalbraith} $\abs{\slp} = \abs{\Gamma(\glf)} = h(-p)$ since $p \equiv 7 \pmod{8}$. It is also known that for any component of $\glf$ the number of vertices on the surface, similarly the number of vertices on the floor, is the order of $\ell$ in the class group of $\Q(\sqrt{-p})$ (results stated in \cite[\nopp Theorem 7]{SutherlandIsogenyNotes}).

\begin{corollary}
    There are $\frac{h(-p)}{\ord(\ell)}$ components in $\glf$. 
\end{corollary}

\begin{proof}
    Recall $\abs{\slp} = h(-p)$. Since $\Gamma$ is a 2-to-1 function on the set of vertices there must be $2h(-p)$ vertices in $\glf$. On the other hand, since each component of $\glf$ contains $\ord(\ell)$ vertices on its surface, and $\ord(\ell)$ vertices on its floor, each component has $2\ord(\ell)$ vertices. Thus, the number of components in $\glf$ must be $\frac{2h(-p)}{2\ord(\ell)} = \frac{h(-p)}{\ord(\ell)}$.
\end{proof}

\begin{figure}[H]
    \centering
    \includegraphics[scale=0.25]{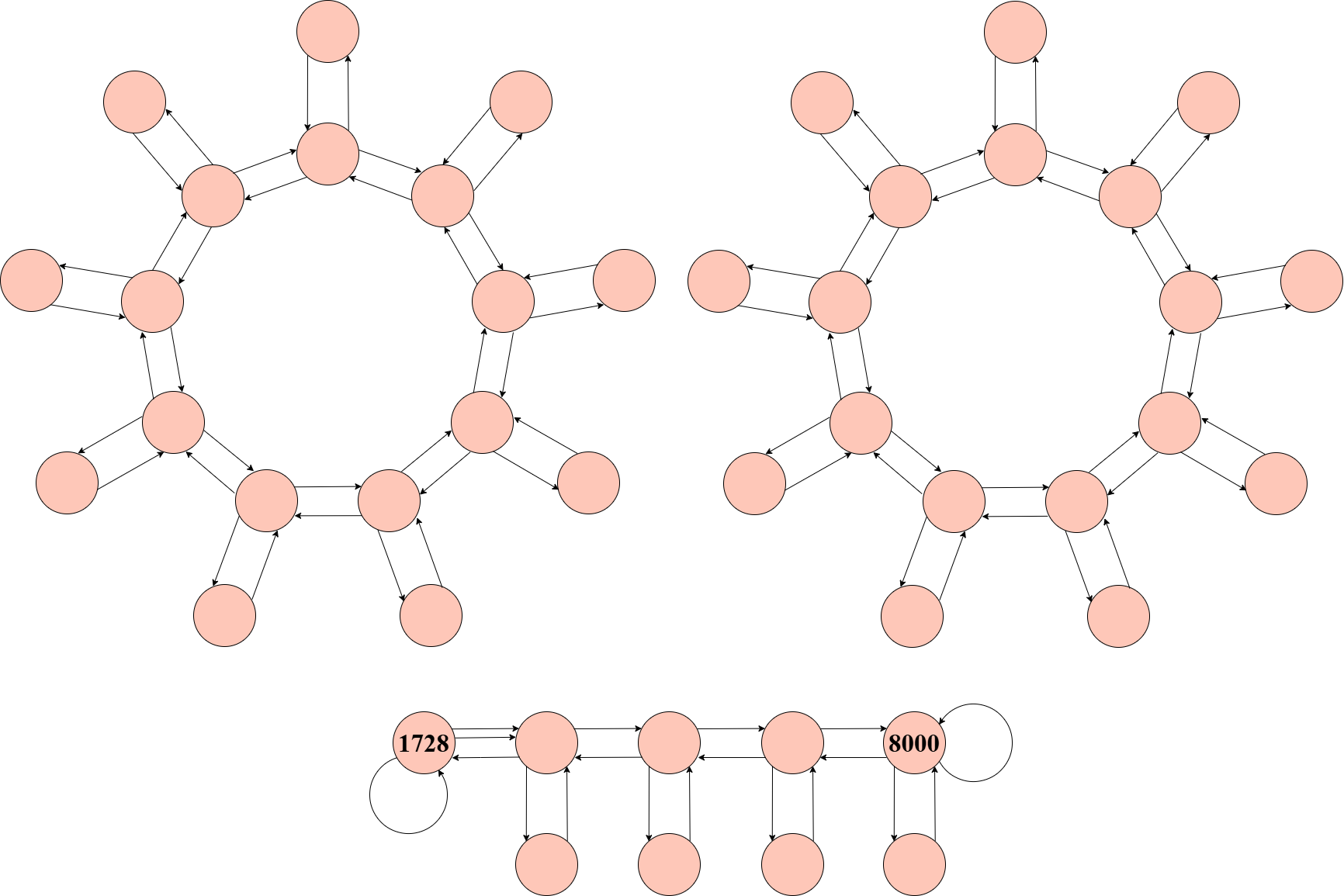}
    \caption{A general description of $\Gamma(\glf)$, when $p \equiv 71, 119 \pmod{120}$ and $\ell =2$, with 1728 and 8000 labeled. The volcano graphs need not exist, but the folded component must exist.}
    \label{fig:Gamma(G_2(F_p))withp=71_119_mod120}
\end{figure}

\begin{corollary}\label{cor:numberOfComponentsOfGammaGlf}
    There are $\frac{1}{2}\brac*{\frac{h(-p)}{\ord(\ell)} +1}$ components in $\Gamma(\glf)$.
\end{corollary}

\begin{proof}
    By \cite[\nopp Proposition 2.1]{HedayatArpinScheidler} $h(-p)$ is odd, hence $\frac{h(-p)}{\ord(\ell)}$ is odd. We know by \cite[\nopp Proposition 4.2]{HedayatArpinScheidler} that $\Gamma$ folds only one component, which means $\frac{h(-p)}{\ord(\ell)}-1$ components stack. As such there are $\frac{1}{2}\brac*{\frac{h(-p)}{\ord(\ell)}-1}$ stacked components in $\Gamma(\glf)$. By adding in the folded component, the number of connected components in $\Gamma(\glf)$ is
    $$1 + \frac{1}{2}\brac*{\frac{h(-p)}{\ord(\ell)}-1} = \frac{1}{2}\brac*{\frac{h(-p)}{\ord(\ell)}+1}$$
\end{proof}
\vspace{-30pt}
\begin{corollary}
    $\abs{\Fcal} = \ord(\ell)$
\end{corollary}

\begin{proof}

    Let $C \leq \glf$ be the unique component which is folded by $\Gamma$. As discussed earlier, there are $\ord(\ell)$ vertices on the surface of $C$ and $\ord(\ell)$ vertices on the floor of $C$. Since $\Gamma$ is a 2-to-1 function on vertices and the vertices in $C$ fold onto other vertices in $C$, 
    $$\abs{\Fcal} = \abs*{\Gamma(C)} = \frac{1}{2}\abs{C} = \frac{1}{2} (2\ord(\ell)) = \ord(\ell)$$
\end{proof}

\vspace{-10pt}On the other hand, since $\Gamma$ does not change the count of vertices for stacked components we get for free that:

\begin{corollary}
    $\abs{\Tcal} = 2\ord(\ell)$
\end{corollary}

In terms of the isogenies in $\glf$ and what happens to them when $\Gamma$ is applied, we will refer to \cite[Corollary 3.7]{Adventures} which we restate here:
\begin{lemma}\label{lem:quadTwistIsogenyEquiv}\textit{(\cite[Corollary 3.7]{Adventures})}
    Let $p > 3$ be a prime, let $E$ be an elliptic curve defined over $\F_p$ and let $E^t$ denote its quadratic twist. Then $\End_{\F_p}(E) \cong \End_{\F_p}(E^t)$.
\end{lemma}

The key consequence of this lemma is that vertices and their quadratic twists possess the ``same'' isogenies with the ``same'' kernels. Crucially, this does not mean that the isogenies from $E$ are equivalent over $\F_p$ to those from $E^t$; in fact, they are not. Instead, $\F_p$ isogenies coming out of $E$ are equivalent, over $\overline{\F_p}$, to the $\F_p$ isogenies coming out of $E^t$. The distinction being made is that the isomorphisms between $E$ and $E^t$ are not defined over $\F_p$ and hence a commuting diagram alike the one in Definition \ref{def:equivIsogeny} cannot be made when we are working over $\F_p$. However, when we do have a way to go from $E$ and $E^t$ (i.e. when we are working over $\overline{\F_p}$) the result of $\End_{\F_p}(E) \cong \End_{\F_p}(E^t)$ guarantees us that we will be able to find isomorphisms which make any $\F_p$ isogeny out of $E$ and $E^t$ equivalent over $\overline{\F_p}$. 

As always, the curve with $j$-invariant 1728 is a significant exception. Since $p \equiv 3 \pmod{4}$, 1728 is a supersingular $j$-invariant. However, its twist is quartic, not quadratic. As such Lemma \ref{lem:quadTwistIsogenyEquiv} does not hold for 1728 and we must explore it ourselves to determine whether or not the isogenies coming out of 1728 become equivalent or not. To this end let us denote $E_{1728}:y^2=x^3-x$ and $E_{1728}^t:y^2 = x^3+4x$. A quick computation shows that these two curves both have the $j$-invariant 1728 and are supersingular. However, note that the isomorphism which would take you from one curve to the other is of the form $(x,y)\mapsto (u^2 x, u^3 y)$ for $u = \sqrt[4]{-4}$. Since $p \equiv 3 \pmod{4}$, $\sqrt{-1}$ is not a quadratic residue. As such $(\sqrt[4]{-4})^2$ is not in $\F_p$, and hence the isomorphism is not defined over $\F_p$. Thus, we conclude that $E_{1728}$ and $E_{1728}^t$ are not isomorphic over $\F_p$. 

By setting $y=0$, we find the 2-torsion points of the curves are
$$E_{1728}[2] = \set{\Ocal, (0,0), (-1,0), (1,0)}\hspace{20pt} E_{1728}^t[2] = \set{\Ocal, (0,0), (-\sqrt{-1}, 0), (\sqrt{-1},0)}$$
where $\Ocal$ is the point at infinity. However, since $\sqrt{-1}$ is not in $\F_p$ and has minimal degree 2, the $p$-power Frobenius endomorphism will not preserve the points $(-\sqrt{-1}, 0)$ and $(\sqrt{-1},0)$ but rather swap them. Thus, the isogenies made by those kernels are not defined over $\F_p$. Let us denote $\phi:E_{1728} \rightarrow E_{1728}^t$ for the 2-isogeny with kernel $\gen{(0,0)}$ and let $\hat{\phi}:E_{1728}^t\rightarrow E_{1728}$ be its dual. Note that, $\ker(\hat{\phi}) = \gen{(0,0)}$ since there exists only one 2-isogeny defined over $\F_p$ coming out of $E_{1728}^t$. For brevity, let us denote $\sqrt{-1}$ by $i$. 

\begin{proposition}\label{prop:1728LoopsMerge}
    With the notation above, $\phi \sim \hat{\phi}$ over $\F_{p^2}$.
\end{proposition}
\begin{proof}
    Take $\F_{p^2} = \sfrac{\F_p[x]}{(x^2+1)}$. In this way, $\sqrt[4]{-4} \in \F_{p^2}$ and we attain all the isomorphisms between $E_{1728}$ and $E_{1728}^t$. $\sqrt[4]{-4}$ could be any of the elements in the set $\set{1+i, -1+i, -1-i, 1-i}$. Hence for any version of $u=\sqrt[4]{-4}$ the map $(x,y)\mapsto (u^2x, u^3 y)$ is of the form $\lambda:E_{1728}\rightarrow E_{1728}^t$. The respective inverses are of the form $(x,y)\mapsto ((\frac{1}{u})^2 x, (\frac{1}{u})^3 y)$. 

    Consider the map $\lambda_1:E_{1728}\rightarrow E_{1728}^t$ made from $u = 1-i$, and $\lambda_2^{-1}:E_{1728}^t \rightarrow E_{1728}$ made from $\frac{1}{u} = \frac{1}{1+i}$. This produces the following diagram of maps:
    \[\begin{tikzcd}
	   {E_{1728}} && {E_{1728}^t} \\
	   \\
	   {E_{1728}^t} && {E_{1728}}
	   \arrow["{\lambda_1}", from=1-1, to=1-3]
	   \arrow["\phi"', from=1-1, to=3-1]
	   \arrow["{\hat{\phi}}", from=1-3, to=3-3]
	   \arrow["{\lambda_2^{-1}}"', from=3-1, to=3-3]
    \end{tikzcd}\]
    By computations in SageMath \cite{sagemath}, we compute
    $$\begin{array}{rcl}
        \phi(x,y)           &= & \left(\frac{x^{2} - 1}{x}, \frac{x^{2} y + y}{x^{2}}\right)\\
        \hat{\phi}(x,y)     &= & \left(\frac{x^{2} + 4}{4x}, y\frac{x^{2} + 4}{8x^{2}}\right)\\
        \lambda_1(x,y)      &= & \brac*{-2ix, -2(1+i)y}\\
        \lambda_2^{-1}(x,y) &= & \brac*{\frac{-i}{2}x, \frac{-1}{4}(1+i)y}
    \end{array}$$
    It can be computed (in SageMath \cite{sagemath}) that the diagram above is commutative. Thus, by definition, we conclude that $\phi$ and $\hat{\phi}$ are equivalent to each other over $\F_{p^2}$.
\end{proof}

\begin{proposition}\label{prop:isogeniesGoingOutOf1728NotLoopDontFold}
    The isogenies going out of $E_{1728}:y^2= x^3-x$ with kernel $\gen{(-1,0)}$ and $\gen{(1,0)}$ are not equivalent over $\overline{\F_p}$. 
\end{proposition}
\begin{proof}
    Recall that our definition of equivalent isogenies depends on post-composition. Since both isogenies are going out of the exact same elliptic curve they cannot be equivalent unless they are the same function. Thus, since the two isogenies have distinct kernels, we conclude that they are not equivalent over $\overline{\F_p}$ even though they both have codomain with $j$-invariant 287496.
\end{proof}

Note that, we have proven the following result:

\begin{lemma}\label{lem:foldingEdges}
    For any $p\equiv 71,119 \pmod{120}$ and $\ell=2$, $\Gamma$ is a 2-to-1 function on the edges of $\glf$, except for the two isogenies going from 1728 to 287496 which do not fold onto each other. 
\end{lemma}

\section{Understanding $\Gamma(\glf)$ for $\ell = 2$ and $p \equiv 71, 119 \pmod{120}$}\label{sec:UnderstandingGammaGlf}

While \cite{HedayatArpinScheidler} did provide some very explicit descriptions of the spine for $\ell = 2$ and $p \equiv 71, 119$\ $\pmod{120}$, the analysis of the diameter of $\slp$ in that paper was incomplete. This was due to the new edge being added. Understanding the change in diameter that occurs with this new edge can allow us to determine when the new edge is connecting two components of $\Gamma(\glf)$ or not. In order to determine the diameter of $\slp$ we must understand the diameter of $\Gamma(\glf)$ and what happens to it when $\Theta$ is applied to it. As such we find the distance, eccentricity, and diameter functions for the components of $\Gamma(\glf)$ to see how they change when $\Theta$ is applied to our graph. We will analyze the components $\Tcal$ and $\Fcal$ separately, since these are the only structures that appear.

Since all graphs under consideration are symmetric and strongly connected, we may treat them as simple graphs (i.e., without loops, multi-edges, or directed edges) when computing distance and eccentricity. As such we can adjust our graphs according to this.

\subsection{Distance and eccentricity of $\Fcal$}\label{subsubsec:distAndEccOfFoldedComp}

Notice that $\Fcal$ is a tree graph by definition (it is connected and has no cycles). Recall the following simple exercise about trees:
\begin{lemma}\label{lemm:TreeHasUniquePath}
    A graph $G$ is a tree if and only if for any two vertices, there exists a unique path connecting them.
\end{lemma}
\begin{proof}
    If there existed two distinct paths between vertex $v$ and $w$ then we could put them together to get a path from $v$ to itself which would violate the ``no cycles'' part of the definition of a tree.
\end{proof}

For the remainder of this article we will relabel $\Fcal$ to be the graph with vertex set
$$V(\Fcal) = \set{0, 1, 2, \ldots, \ord(\ell)-1}$$
where $0$ is the new label for the isomorphism class of elliptic curves with $j$-invariant $1728$, and $\ord(\ell) - 2$ is the new label for the isomorphism class of elliptic curves with $j$-invariant $8000$. Furthermore, we set the edge set of $\Fcal$ to be 
$$E(\Fcal) = \set{\set{0, 1}, \set{1, 2}, \set{1, 3}, \set{3, 4}, \set{3, 5}, \ldots, \set{\ord(\ell)-2, \ord(\ell)-1}}$$
Or in other words, we set the edges of $\Fcal$ to be 
$$E(\Fcal) = \set{\set{2k - 1 , 2k+1}}_{k=1}^{\frac{1}{2}(\ord(\ell)-3)} \cup \set{\set{2k+1, 2k+2}}_{k=0}^{\frac{1}{2}(\ord(\ell)-3)} \cup \set{\set{0,1}}$$
An explicit example of what this relabeling is provided in Figure \ref{fig:foldedCompWithNewLabelOrd15}.

\begin{figure}
    \centering
    \includegraphics[width=0.5\linewidth]{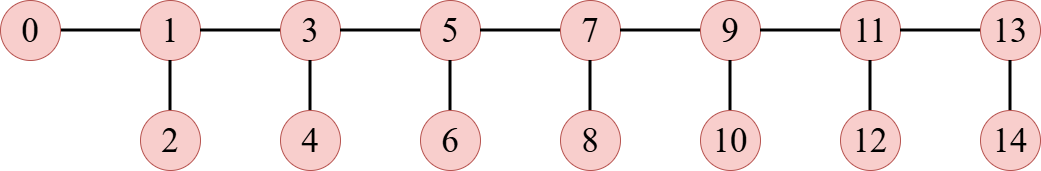}
    \caption{A relabeled folded component of $\Gamma(\glf)$ where $\ord(\ell) = 15$.}
    \label{fig:foldedCompWithNewLabelOrd15}
\end{figure}

Notice that in this configuration we see that for any $v \in V(\Fcal)$, $v$ is on the surface if and only if $2 \nmid v$ or $v=0$, and $v$ is on the floor if and only if $2 | v$ and $v\neq 0$. In other words, if we exclude 0, then evens are on the floor and odds are on the surface. In fact, we can exclude 0 when finding a distance function! This is shown in the following proposition:

\begin{proposition}\label{prop:Distance0EqualsDistance2}
    Let $v \in V(\Fcal)$, then 
    \begin{enumerate}
        \item $d(0,v) = d(2,v)$ if $v \in V(\Fcal) \setminus \set{0,2}$
        \item $d(0,v) = d(2,v)-2$ if $v=0$
        \item $d(0,v) = d(2,v)+2$ if $v=2$
    \end{enumerate}

\end{proposition}

\begin{proof}
    The cases $v=0$ and $v=2$ follow from $d(0,0) = 0 \neq 2 = d(2,0)$ and $d(2,2) = 0 \neq 2 = d(0,2)$. Thus, by simply adding or subtracting 2, we get the desired result.

    Now, suppose $v \neq 0,2$. In such a case there exists a path $P_0$, going from $0$ to $v$, and $P_2$, going from 2 to $v$, both having minimal length. Meaning that $d(0,v) = \abs{P_0}$ and $d(2,v) = \abs{P_2}$. Note that, $N(0) = N(2) = \set{1}$. As such, since $v\neq 0,2$ any path between $v$ and $0$ or $2$ must first reach 1. Hence
    $$P_0 = (0, 1, v_{03}, v_{04}, \ldots, v) \text{ and } P_2 = (2, 1, v_{23}, v_{24}, \ldots, v)$$
    Note that, $P_{01} = (1, v_{03}, v_{04}, \ldots, v)$ and $P_{21} = (1, v_{23}, v_{24}, \ldots, v)$ are paths themselves (because, if they were not paths, then $P_0$ or $P_2$ would not be paths). Thus, by Lemma, \ref{lemm:TreeHasUniquePath} $P_{01} = P_{21}$. As such
    \begin{align*}
        d(0,v) &= \abs{P_0}\\
               &= \abs{(0, 1, v_{03}, v_{04}, \ldots, v)}\\
               &= \abs{(0,1)} + \abs{(1, v_{03}, v_{04}, \ldots, v)}\\
               &= \abs{(0,1)} + \abs{P_{01}}\\
               &= \abs{(0,1)} + \abs{P_{21}}\\
               &= \abs{(0,1)} + \abs{(1, v_{23}, v_{24}, \ldots, v)}\\
               &= \abs{(2, 1, v_{23}, v_{24}, \ldots, v)}\\
               &= \abs{P_2}\\
               &= d(2,v)
    \end{align*}
\end{proof}

\begin{remark}\label{rem:ecc0EqalEcc2}
    If $\ord(\ell) = 3$ then $\ecc(0) = \ecc(2)$ given by the fact that $\Fcal$, in this case, is a path graph with 0 and 2 as its end-vertices. On the other hand, if $\ord(\ell) >3$ (we will see in Proposition \ref{prop:foldedEccEven} that) $\ecc(2) > 2$. Meaning that the eccentricity of 2 is caused by a vertex distinct from 0. As such, by Proposition \ref{prop:Distance0EqualsDistance2} we can conclude that $\ecc(0) = \ecc(2)$ in this case. Thus, in all cases, it stands true that $\ecc(0) = \ecc(2)$.
\end{remark}

Thus, by Proposition \ref{prop:Distance0EqualsDistance2}, for our purposes of finding the distance and eccentricity function of $\Fcal$, we can exclude 0 from $\Fcal$. This will allow us to derive the formulas much more efficiently.

\begin{proposition}\label{prop:DistanceFunctionOfFoldedComponent}
    Let $v,w \in V(\Fcal) \setminus \set{0}$. Then 
    $$d(v,w) = \begin{cases}
        \frac{1}{2} \abs{v-w}       &\text{if } 2 \nmid v \text{ and } 2 \nmid w, \text{ or } v=w\\
        \frac{1}{2} \abs{v-w-1} + 1 &\text{if } 2 | v     \text{ and } 2 \nmid w \\
        \frac{1}{2} \abs{v-w} + 2 &\text{if } 2 | v     \text{ and } 2 | w      \text{ and } v \neq w\\
    \end{cases}$$
\end{proposition}

\begin{proof}
    Since $\Fcal$ is a simple graph, the distance function is symmetric. Hence, the conditions on $d$ are mutually exclusive and complete. Making $d$ a well-defined function.

    Let $v,w \in V(\Fcal) \setminus \set{0}$.

    Suppose $2 \nmid v$ and $2 \nmid w$. In such a case both vertices are on the surface of the graph. As such the shortest path is the one which iterates through the odd values, starting at either $v$ or $w$ and terminating at the other. As such, the unique path connecting $v$ and $w$ is either $P = (v, v+2, \ldots, v+ 2(k-1), v + 2k = w)$ or $P = (v, v-2,\ldots, v-2(k-1), v-2k = w)$ for some $k$. In fact, our labeling forces $k = \frac{1}{2} \abs{v-w}$. Thus, we conclude that $d(v,w) = \abs{P} = \frac{1}{2}\abs{v-w}$. This computations also works for when $v=w$ without the restriction of them both being odd. 

    On the other hand, suppose $2 | v$ and $2 \nmid w$. Since $2|v$, $v$ is on the floor of the graph. As such, $N(v) = \set{v-1}$. Now, since $2|v$, it must be the case that $2\nmid (v-1)$ and hence we can use the previous case to find a path from $v-1$ to $w$. Adding onto that path, the path $(v,v-1)$, we get a path from $v$ to $w$. By Lemma \ref{lemm:TreeHasUniquePath} this path is unique and hence the distance from $v$ to $w$ is equal to the distance from $v-1$ to $w$ plus 1. In other words, 
    $$d(v,w) = 1 + d(v-1, w) = 1 + \frac{1}{2} \abs{(v-1)-w} = \frac{1}{2}\abs{v-w-1} + 1$$

    Lastly, suppose $2|v$ and $2|w$ and $v \neq w$. Similar to the previous case, $v-1$ and $w-1$ exist, are odd, and are the unique vertices adjacent to $v$ and $w$ respectively. Since $v \neq w$, the edge $\set{v, v-1}$ is not equal to the edge $\set{w, w-1}$, and hence if we add these edges to a path, it will continue to stay a path. Note that
    \begin{align*} 
    d(v,w) &= 1 + d(v-1, w) 
           = 1 + d(v-1, w-1) + 1 
           = d(v-1,w-1) + 2 
           = \frac{1}{2}\abs{v-w} + 2
    \end{align*}
\end{proof}

Now that we have the distance function on $\Fcal$ we can try to look at the eccentricity. Recall that the eccentricity of a vertex is the maximal distance of that vertex. To this goal we need to be able to compare paths within the folded component. 

\begin{lemma}\label{lemm:odd_oddPathsAndEven_OddPaths}
    If $P$ is a path between two odd vertices, or between an odd and an even vertex, then there exists a strictly longer path $P'$ containing $P$.
\end{lemma}

\begin{proof}
    Suppose $P \leq \Fcal$ is a path subgraph connecting two odd vertices. Without loss of generality, let $P = (v, v+2, \ldots, v+2k=w)$ for some $v,w \in V(\Fcal)$ such that $2 \nmid v$ and $2 \nmid w$. Since $w$ is odd, $\set{w,w + 1} \in E(\Fcal)$. Thus, we can consider $P' = (v,  v+2, \ldots, v+2k =w, w+1)$. Note that, $P$ is a strict sub-path of $P'$. A similar proof holds for when $P$ is a path between an odd and even vertex. Through some index changes of the above proof, we may extend the result to include the vertex 0.
\end{proof}

Note that, Lemma \ref{lemm:odd_oddPathsAndEven_OddPaths} shows us that the eccentricity of a vertex is always achieved by going to an even vertex. With this notion, along with the distance function of the folded component, we can compute the eccentricity of any vertex efficiently. 

\begin{proposition}\label{prop:foldedEccOdd}
    For all $v \in V(\Fcal)$ such that $2 \nmid v$,
    $$\ecc(v) = \begin{cases}
        \frac{1}{2} (v+1) & \text{if } v \geq \frac{1}{2} (\ord(\ell) -1)\\
        \frac{1}{2} (\ord(\ell) - v) & \text{if } v < \frac{1}{2} (\ord(\ell) -1)
    \end{cases}$$
\end{proposition}

\begin{proof}

    Let $v \in V(\Fcal)$ such that $2 \nmid v$. By definition $\ecc(v) = \max_{w \in V(\Fcal)}d(v,w)$. However, if $2 \nmid w$, then by Lemma \ref{lemm:odd_oddPathsAndEven_OddPaths} a path from $v$ to $w$ will not be the largest path starting at $v$. As such we may restrict to when $2 | w$ and $w \neq 0$. By Proposition \ref{prop:Distance0EqualsDistance2}
    \begin{align*}
        \ecc(v) &= \max_{w \in V(\Fcal), 2|w} d(v,w)\\
        &= \max_{w \in V(\Fcal), 2|w} d(w,v)\\
        &= \max_{w \in V(\Fcal), 2|w} \brac*{\frac{1}{2} \abs{w-v-1}+1}\\
        &= \frac{1}{2}\max_{w \in V(\Fcal), 2|w} (\abs{w-v-1}) + 1
    \end{align*}
    Since $w \in V(\Fcal)$, $1 \leq w \leq \ord(\ell)-1$. In particular, from $2|w$, $2 \leq w \leq \ord(\ell) -1$. As such
    $$\max_{w \in \Fcal, 2|w}(\abs{w-v-1}) = \max (\abs{2-v-1}, \abs{\ord(\ell) - 1 - v - 1}) = \max(\abs{1-v}, \abs{\ord(\ell)-v-2})$$
    Note that, $v \in V(\Fcal)$ and $2 \nmid v$, hence $1 \leq v \leq \ord(\ell)-2$. Which implies, $1 - v \leq 0 \Rightarrow \abs{1-v} = v-1$. It also implies, $0 \leq \ord(\ell) - v - 2\Rightarrow \abs{\ord(\ell) - v-2} = \ord(\ell) -v-2$. Thus
    $$\max(\abs{1-v}, \abs{\ord(\ell) - v - 2}) = \max(v-1, \ord(\ell) - v - 2)$$
    As such we see that
    $$\frac{1}{2}\max_{w \in V(\Fcal), 2|w} (\abs{w-v-1}) + 1 = \frac{1}{2}\max(v-1, \ord(\ell) - v - 2) + 1 = \frac{1}{2} \max(v+1, \ord(\ell) - v)$$
    Note that, 
    $$v+1 = \ord(\ell) -v \Leftrightarrow 2v = \ord(\ell) -1 \Leftrightarrow v = \frac{1}{2}(\ord(\ell)-1) $$
    Thus, when $v = \frac{1}{2}(\ord(\ell)-1)$, the max function swaps its output.
\end{proof}

\begin{lemma}\label{lemm:eccOddToEccEven}
    For any $2k+1 \in V(\Fcal)$, $\ecc(2k+1) + 1 = \ecc(2k+2)$.
\end{lemma}

\begin{proof}
    Let $P$ be a path starting at $2k+1$ of maximal length. If $P = (2k+1, 2k+2)$, then by the maximality of $P$, $\ecc(2k+1) = 1$. This in turn would imply $\ord(\ell) = 3$, otherwise there would exist an even vertex on the floor which has a distance of 2 from $2k+1$. If $\ord(\ell) = 3$, then $\ecc(2k+2) = \ecc(2) = 2$ provided by the distance between 2 and 0. Thus, in this case, the statement holds true. 

    On the other hand, if $P \neq (2k+1, 2k+2)$, then $P = (2k+1, w_1, \ldots, w_n)$ where $w_1, \ldots, w_n \neq 2k+2$ (otherwise $P$ would not be a path). Since $\set{2k+1, 2k+2} \in E(\Fcal)$, $P' := (2k+2, 2k+1, w_1, \ldots, w_n)$ is a valid path from $2k+2$ to $w_n$. If $d(2k+2, w_n) \neq \ecc(2k+2)$, then there would exists some vertex $u$ such that $d(2k+2, u) > d(2k+2, w_n)$. However, 
    $$\abs{P} = \abs{P'}-1 = d(2k+2, w_n) - 1 < d(2k+2, u) -1  = 1 + d(2k+1, u) -1 = d(2k+1,u)$$
    violating the maximality of $P$.

\end{proof}

\begin{proposition}\label{prop:foldedEccEven}
    Let $v \in V(\Fcal)\setminus \set{0}$ such that $2 | v$. Then 
    $$\ecc(v) = \begin{cases}
        \frac{1}{2} v + 1 &\text{if } v \geq \frac{1}{2} (\ord(\ell)+1)\\
        \frac{1}{2}(\ord(\ell) - v + 3) &\text{if } v < \frac{1}{2} (\ord(\ell)+1)
    \end{cases}$$
\end{proposition}

\begin{proof}
    By Lemma \ref{lemm:eccOddToEccEven} if $v > 0$ and $2|v$ then $\ecc(v) = \ecc(v-1) + 1$. 

    The computation results in precisely the statement wanted.
\end{proof}

\subsection{Distance and eccentricity of $\Tcal$}\label{subsubsec:distAndEccOfStackedComp}

In the components which are stacked, we no longer have unique paths connecting two vertices due to the cycle on the surface. However, we still have a finite number of ways to traverse the graph between any two vertices, so it is not too hard to find and compute the eccentricity and diameter of the graph $\Tcal$.

\begin{lemma}\label{lemm:twoPathsOnlyInStacked}
    Let $G$ be a graph with $V(G) = \set{v_1, v_2, \ldots, v_{n}, v_1', v_2', \ldots, v_{n}'}$, where $n \geq 3$, and $E(G) =\\ \set{\set{v_i, v_i'}}_{i=1}^{n} \cup \set{\set{v_i, v_{i+1}}}_{i=1}^{n-1} \cup \set{\set{v_{n}, v_1}}$ (i.e. consider a graph isomorphic to a stacked component). Then for any distinct $v,w \in V(G)$, there exist precisely two distinct paths from $v$ to $w$, unless $v=v_j$ and $w = v_j'$ (or vice-versa) for some $j$, at which point there would exist a unique path between $v$ and $w$.
\end{lemma}

\begin{proof}
    Let $G$ be as described in the statement of the lemma and let $v,w \in V(G)$ be distinct vertices.
    
    \underline{Case 1:} Suppose $v = v_j$ and $w = v_j'$ for some $j$. Then $\set{v,w} \in E(G)$. As such $P_1 = (v,w)$ is a valid path from $v$ to $w$. Further note that $N(w) = \set{v}$ hence any path that ends at $w$ must first reach $v$. As such $P_1$ must be a sub-path of any path from $v$ to $w$. Thus, if $P_2$ is any other path from $v$ to $w$, then $P_2 = (v, v^{(1)}, v^{(2)}, \ldots, v, w)$. However this is not a path. Thus, we conclude that any path from $v$ to $w$ must be of the form $(v,w)$, making the path unique.

    \underline{Case 2:} Suppose that there does not exists any $j$ such that $v=v_j$ and $w = v_j'$ (or vice-versa). 
    
    If it is the case that $v = v_k' \in \set{v_i'}_{i=1}^{n}$ for some $k$, then any path from $v$ to $w$ must be of the form $P=(v = v_k', v_k, \ldots, w)$ since $N(v_k') = v_k$. $P$ has a unique sub-path starting at $v_k$; $P' = (v_k, \ldots, w)$. As such, by only considering paths starting with $v_k$, we uniquely identify paths starting at $v_k'$. For this reason, we may assume without loss of generality, there exists $1 \leq j < k \leq n$ such that $v = v_j$ and $w = w_k$. 

    Let 
    $$P = \brac*{v_j = v^{(0)}, v^{(1)}, v^{(2)}, \ldots, v^{(d-1)}, v^{(d)} = v_k}$$
    be a path from $v$ to $w$, where $d >0$. If it is the case that for some $1 \leq r \leq d-1$, $v^{(r)} \in \set{v_i'}_{i=1}^{n}$, then $v^{(r-1)} = v^{(r+1)}$ since elements of $\set{v_i'}_{i=}^{n}$ only have one neighbour. This violates the definition of a path. Hence it must be the case that for all $0 \leq m \leq d$, $v^{(m)} \notin \set{v_i'}_{i=1}^n$. Given by the fact that the induced subgraph of $G$ made of the vertices $\set{v_i}_{i=1}^n$ is a cycle (and knowing that there only exists two paths between two distinct vertices on a cycle), we conclude that there only exists two paths from $v$ to $w$. 
    
\end{proof}

For the sake of readability and understandability we will relabel the stacked component, $\Tcal$. We do this by picking a random vertex on the graph with degree 1 and labeling it 0. The vertex it is adjacent to will be labeled 1. Then we go around the connected component by putting even vertices on the floor and odd ones on the surface. An example of $\Tcal$ with $\ord(\ell) = 13$ is shown in Figure \ref{fig:relabeledStackedComponent}.

\begin{figure}
    \centering
    \includegraphics[scale=0.2]{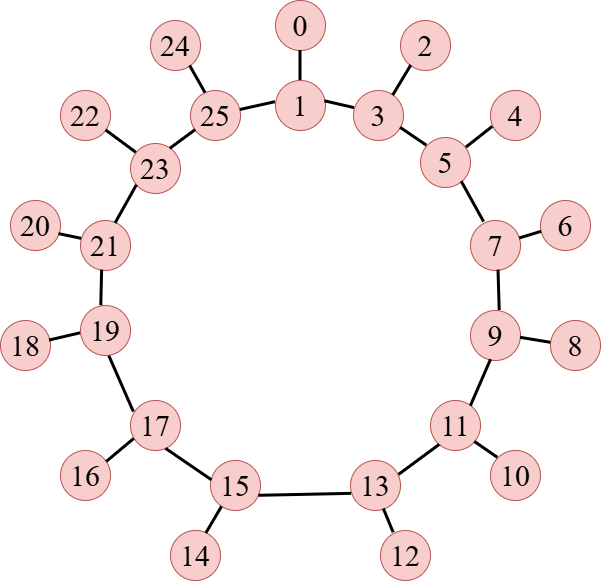}
    \caption{A relabeled stacked component of the graph $\Gamma(\glf)$ where $\ord(\ell) = 13$.}
    \label{fig:relabeledStackedComponent}
\end{figure}

\begin{proposition}\label{prop:stackedDistanceFunction}
    Let $v, w \in V(\Tcal)$ be distinct vertices. If $2\nmid v$ and $2 \nmid w$ then 
    $$d(v,w) = \begin{cases}
        \frac{1}{2} \abs{v-w} &\text{if } \abs{v-w} \leq \ord(\ell)-1\\
        \ord(\ell) - \frac{1}{2} \abs{v-w} &\text{if } \abs{v-w} \geq \ord(\ell)+1\\
    \end{cases}$$
    If $2|v$ and $2\nmid w$ then $d(v,w) = 1 + d(v+1, w)$, and if $2|v$ and $2|w$ then $d(v,w) = 2 + d(v+1, w+1)$. 
\end{proposition}

\begin{proof}
    Note that, if we prove the statement for the case when $2\nmid v$ and $2\nmid w$ then, by the new labeling of the stacked component, the other two cases follow suit. 

    Suppose $2 \nmid v$ and $2 \nmid w$. In addition suppose that $v+2k = w$ for some $k\in \Z_{\geq 0}$. By Lemma \ref{lemm:twoPathsOnlyInStacked} the paths $P_1(v, v+2, \ldots, v+2k = w)$ and $P_2=(v, v-2, \ldots, 1, 2\ord(\ell) - 1, \ldots, w)$ are the only paths from $v$ to $w$.

    Notice that 
    $$\abs{P_1} = \frac{1}{2}(w-v) = \frac{1}{2} (v+ 2k - v) = \frac{1}{2}(2k)= k$$
    On the other hand
    \begin{align*}
        \abs{P_2} &= \abs{(v, v-2, \ldots, 1, 2\ord(\ell) - 1, \ldots, w)}\\
        &= \abs{(v, v-2, \ldots, 1)} +\abs{(1, 2\ord(\ell) - 1)}  + \abs{(2\ord(\ell) - 1, \ldots, w)}\\
        &= \frac{1}{2}(v-1) + 1 +  \frac{1}{2}(2\ord(\ell) - 1 - w)\\
        &= \ord(\ell) - k
    \end{align*}
    Note that, $k = \frac{1}{2}(w-v)$. If it were the case that $w +2k = v$ then we would get the same lengths stated above, but instead have $k = \frac{1}{2}(v-w)$. As such, to merge the two cases together, we get $k = \frac{1}{2}\abs{v-w}$. 

    Thus, by the definition of a distance 
    $$d(v,w) = \min\brac*{\frac{1}{2}\abs{v-w}, \ord(\ell) - \frac{1}{2}\abs{v,w}}$$
    Notice that $\frac{1}{2}\abs{v-w} =  \ord(\ell) - \frac{1}{2}\abs{v,w}$ if and only if $\abs{v-w} = \ord(\ell)$. However, this equality will never hold in our case, due to the parity of $v$, $w$, and $\ord(\ell)$. As such by subtracting one from $\ord(\ell)$ in one case of the condition and adding one to $\ord(\ell)$ in the other case, we get the result stated.
\end{proof}

\begin{proposition}\label{prop:stackedEccFunction}
    Let $v \in V(\Tcal)$. Then 
    $$\ecc(v) = \begin{cases}
        \frac{1}{2}(\ord(\ell) + 1) & \text{if } 2\nmid v\\
        \frac{1}{2}(\ord(\ell) + 3)& \text{if } 2\mid v\\
    \end{cases}$$
\end{proposition}
\begin{proof}
    Suppose $v \in V(\Tcal)$. By a similar argument as the one for Lemma \ref{lemm:odd_oddPathsAndEven_OddPaths}, we may restrict ourselves to look at the distance between $v$ and vertices (distinct from $v$) on the floor to find the eccentricity of $v$.

    \underline{Case 1:} Suppose $2 \nmid v$. We wish to compute $\ecc(v) = \max_{w \in V(\Tcal)}d(v,w)$. If $2 \mid w$ then $w$ is on the floor of $\Tcal$ and hence there exists a unique vertex which $w$ is adjacent to; this vertex lies on the surface. Thus, if we compute the distance between $v$ and that surface vertex, then the distance to $w$ is given by adding 1. As such, consider $\max_{2 \nmid w} d(v,w)$. By Proposition \ref{prop:stackedDistanceFunction} we compute
    \begin{align*}
        \max_{2 \nmid w} d(v,w) &= \max\begin{cases}
        \frac{1}{2} \abs{v-w} &\text{if } \abs{v-w} \leq \ord(\ell)-1\\
        \ord(\ell) - \frac{1}{2} \abs{v-w} &\text{if } \abs{v-w} \geq \ord(\ell)+1\\
    \end{cases}
    = \frac{1}{2}(\ord(\ell) -1)
    \end{align*}
    Thus, we conclude that
    $$\ecc(v) = \max_{w} d(v,w) = 1 + \max_{2\nmid w}d(v,w) = \frac{1}{2}(\ord(\ell) + 1)$$
    \underline{Case 2:} Suppose $2 \mid v$. By using the previous case, we have
    $$\ecc(v) = \max_w d(v,w) = 1 + \max_w d(v+1,w) = \frac{1}{2}(\ord(\ell)+3)$$
\end{proof}

\subsection{Diameter of $\Fcal$ and $\Tcal$}

\begin{theorem}\label{thrm:foldedCompDiameter}
    $\diameter(\Fcal) = \frac{1}{2}(\ord(\ell) + 1)$
\end{theorem}

\begin{proof}
    By Lemma \ref{lemm:eccOddToEccEven}, if we consider vertices with even label then we will have a higher eccentricity value. Thus, 
    $$\diameter(\Fcal) = \max_v \ecc(v) = \max_{2\mid v} \ecc(v)$$
    As such, from Remark \ref{rem:ecc0EqalEcc2} and Proposition  \ref{prop:foldedEccEven}
    \begin{align*}
        \max_{2\mid v} \ecc(v) &= \max_{2 \mid v, v \neq 0}\ecc(v)
        = \max \begin{cases}
            \frac{1}{2} v + 1 &\text{if } v \geq \frac{1}{2} (\ord(\ell)+1)\\
            \frac{1}{2}(\ord(\ell) - v + 3) &\text{if } v < \frac{1}{2} (\ord(\ell)+1)
        \end{cases}
        = \frac{1}{2}(\ord(\ell) + 1)
    \end{align*}

\end{proof}

\begin{theorem}\label{thrm:stackedCompDiameter}
    For any $\Tcal$, $\diameter(\Tcal) = \frac{1}{2}(\ord(\ell) + 3)$
\end{theorem}

\begin{proof}
    By definition, $\diameter(\Tcal) = \max_v \ecc(v)$. By Proposition \ref{prop:stackedEccFunction}, when $2 \mid v$ we will have a higher value of eccentricity. Thus
    $$\diameter(\Tcal) = \max_v\ecc(v) = \max_{2 \mid v}\ecc(v) = \max_{2\mid v} \frac{1}{2}(\ord(\ell) + 3) = \frac{1}{2}(\ord(\ell) + 3)$$

\end{proof}

\section{Understanding $\Theta(\Gamma(\glf))$ for $\ell = 2$ and $p \equiv 71, 119 \pmod{120}$}\label{sec:understandingThetaGammaGlf}

We know by \cite[\nopp Proposition 4.4]{HedayatArpinScheidler} that $\Theta$ adds a new edge between the $j$-invariants corresponding to the roots of the Hilbert class polynomial
$$H_{-15}(x) = x^2 + 191025x - 121287375$$
The roots are of the form
$$\frac{-191025 \pm 85995\sqrt{5}}{2}$$
However, since $\sqrt{5}$ is the root of the irreducible polynomial $x^2 + 5$ (in $\Q$) we know by \cite{toth} that the numerical value of $\sqrt{5} \in \F_p$ is randomly distributed between $1$ and $p-1$. As such we cannot determine the explicit value of $\sqrt{5}$ a priori. Hence, we cannot determine where this new edge ends up. In fact, we expect it to behave randomly. In order to confirm this belief we will look at what happens to the diameter of $\Gamma(\glf)$ when $\Theta$ is applied to it both theoretically and computationally to determine the likelihood of edge attachments occurring.

For the next sections we will look at what happens to components $\Tcal$ and $\Fcal$ when a non-attaching edge occurs in them. Followed by looking at when an attaching edge appears between the folded component and a stacked component; and similarly, when edge attachment occurs between two stacked components. It will become apparent that computations involving the folded component are challenging. As such we will first look at when the new edge does not interact with the folded component, then move onward to looking at when it does. 

Let $\Ecal = \set{v_1, v_2}$ be the new edge added by $\Theta$ to $\Gamma(\glf)$. The following result was implicitly stated in \cite{HedayatArpinScheidler}, however we will explicitly state it here.
\begin{lemma}\label{lem:newEdgeIsOnFloor}
    Both $v_1$ and $v_2$ are on the floor of $\slp$ and $\Gamma(\glf)$.
\end{lemma}

\begin{proof}
    As shown in \cite{HedayatArpinScheidler}, when $p \equiv 71, 119 \pmod{120}$ and $\ell = 2$, the component containing 1728 and 8000 folds, and all other components stack. By Lemma \ref{lem:foldingEdges} any vertex on the surface (including 1728) of $\Gamma(\glf)$ has three outgoing edges. Hence, a new edge cannot become incident to them, because in such a case we would be violating the $\ell+1$ (out-)regularity of $\glfb$.
\end{proof}

\subsection{$\Ecal$ Not an Attaching Edge}\label{subsec:EcalNotAnAttachingEdge}

Let us keep the same notation of $\Tcal$ and $\Fcal$ as in Section \ref{sec:UnderstandingGammaGlf}. We will use this convenient labeling to understand what the new edge $\Ecal$ does to our graph. However, before we can look at our particular graph $\Gamma(\glf)$ that we are adding an edge to, we must consider a graph theoretic result which we will use. 

\begin{lemma}\label{lem:nonAttachEdgeLowerDiamGraphTheory}
    Let $G$ be any simple graph (without loops, multi-edge, or directed edges). If there exists $v,w \in V(G)$ such that $w \notin N(v)$ and $w \neq v$, then $\diameter(G) \geq \diameter(G + e)$ where $e$ is the edge $\set{v,w}$.
\end{lemma}

\begin{proof}
    If $G$ is disconnected, then $\diameter(G) = \infty$ making $\diameter(G+e) \leq \infty$. Hence the statement holds true in this case. 
    
    Let $G$ be a connected graph. Suppose, in hopes of reaching a contradiction, $\diameter(G) < \diameter(G+e)$. By definition there exists $u \in V(G)$ such that $\ecc_G(u) < \ecc_{G + e}(u)$. Thus, there exists $x \in V(G)$ such that $d_G(u,x) < d_{G + e}(u,x)$. Let $P_G$ be a path of minimal length in $G$ between $u$ and $x$, i.e. $\abs{P_G} = d_G(u,x)$. Note that, by definition $P_G$ is a path in $G+e$ (since no edges nor vertices were removed). As such $d_{G+e}(u,x) \leq \abs{P_G} = d_G(u,x)$, implying the contradiction $d_G(u,x) < d_G(u,x)$.
\end{proof}

Notice that the importance of Lemma \ref{lem:nonAttachEdgeLowerDiamGraphTheory} is in helping us have brevity within the proof in the remainder of this subsection. In particular, the lemma tells us that if we find two vertices $v,w \in G$ such that $\diameter(G) = d_{G}(v,w) = d_{G+e}(v,w)$ then the diameter of the graph has not changed (since the diameter cannot increase, and there exists a distance of the maximal value). This will come in handy in the proof of the following theorem:

\begin{theorem}\label{thrm:diameterOfStackedComponentWithNonAttachingEdge}
    For any $\Tcal$, if $\Ecal\in E(\Theta(\Tcal))$ then $\diameter(\Tcal) = \diameter(\Theta(\Tcal))$.
\end{theorem}

\begin{proof}
    Let $\Tcal$ be a stacked component and suppose that $\Ecal \in E(\Theta(\Tcal))$ (i.e. the new edge is not an attaching edge and is contained in a stacked component). By Lemma \ref{lem:newEdgeIsOnFloor}, $\Ecal$ is incident to two vertices on the floor. Hence (through our labeling system) $\Ecal = (2k,2l)$ for some $k$ and $l$. However, since $\Tcal$ is a cycle-like graph, through the function which relabels $\Tcal$ by taking a label $v$ and mapping it to $v - 2k$ we get an isomorphism. Hence, without loss of generality, we can assume that $\Ecal = (0,2k)$ for some $0 \leq k \leq \ord(\ell)-1$. We can simplify this further through the symmetry that $\Tcal$ has. Through the function $g$, which relabels $\Tcal$ again by
    $$g(v) = \begin{cases}
        v &\text{if } v = 0,1\\
        2\ord(\ell) - v &\text{if } 2|v \text{ and } v\neq 0,1\\
        2\ord(\ell) - v + 1 &\text{if } 2\nmid v \text{ and } v\neq 0,1
    \end{cases}$$
    to get an isomorphism. Note that, in this case, if $k \geq \frac{1}{2}(\ord(\ell)+1)$, then $g(2k) \leq \ord(\ell)-1$. Hence without loss of generality we can assume that $k \leq \frac{1}{2}(\ord(\ell) - 1)$. An important thing to note is that in our case of $p\equiv 71, 119 \pmod{120}$, \cite[\nopp Theorem 4.8]{HedayatArpinScheidler} proved that $\Ecal$ cannot be a loop. Hence we can assume that $\Ecal = (0,2k)$ with $1 \leq k \leq \frac{1}{2}(\ord(\ell) - 1)$. 

    \underline{Case 1:} Suppose $k=1$ and $\ord(\ell) = 3$. Hence $\Ecal = (0,2)$. In such a case $\Theta(\Tcal)$ looks like Figure \ref{fig:k=1OrdL=3}, and we can see that $d(0,4) = 3$. Hence, by Lemma \ref{lem:nonAttachEdgeLowerDiamGraphTheory}, the diameter does not change in this case.
    \begin{figure}
        \centering
        \includegraphics[width=0.2\linewidth]{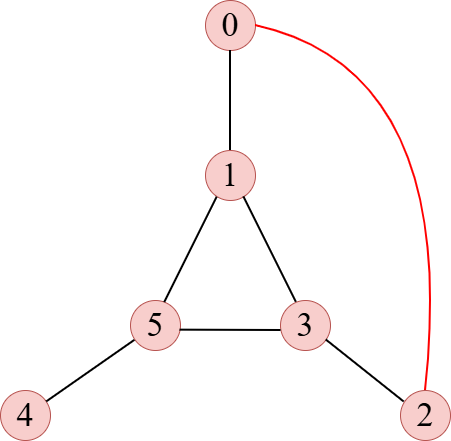}
        \caption{What $\Theta(\Tcal)$ looks like when $\Ecal = \set{0,2}$ and $\ord(\ell) =3$.}
        \label{fig:k=1OrdL=3}
    \end{figure}

    \underline{Case 2:} Suppose $k=1$ and $\ord(\ell) >3$. Recall that by Proposition \ref{prop:stackedDistanceFunction}, $d_{\Tcal}(0, \ord(\ell)-1) = \frac{1}{2}(\ord(\ell) + 3)$. If we are to use the edge $\Ecal$ to find a new minimal path between $0$ and $\ord(\ell)-1$, then it must be the first edge we use in our path. Otherwise we would loop back onto $0$ before reaching $\ord(\ell)-1$. Thus, this path which uses $\Ecal$ must start off as $P = (0,2, \ldots)$. However, $P$ cannot have reached $\ord(\ell)-1$ when it reaches 2, since $\ord(\ell) > 3$. Thus, $P$ must now go from $2$ to $\ord(\ell)-1$. Note that, the rest of $P$ cannot use $\Ecal$ again, otherwise it would violate the definition of a path. As such, the remainder of $P$ is strictly contained in $\Tcal$. By Proposition \ref{prop:stackedDistanceFunction}, $d_\Tcal(2,\ord(\ell)-1) = \frac{1}{2}(\ord(\ell)+1)$. Thus, $\abs{P} = 1 + d_\Tcal(2,\ord(\ell)-1) = \frac{1}{2}(\ord(\ell)+3)$. Hence we see that a minimal path that uses $\Ecal$ to get from $0$ to $\ord(\ell)-1$ will have the same length as a minimal path which does not use $\Ecal$. We have proven, 
    $$d_{\Tcal + \Ecal}(0, \ord(\ell)-1) = d_\Tcal(0, \ord(\ell)-1) = \frac{1}{2}(\ord(\ell)+3) = \diameter(\Tcal).$$
    Thus, by Lemma \ref{lem:nonAttachEdgeLowerDiamGraphTheory}, we conclude that (in this case) $\diameter(\Tcal) = \diameter(\Theta(\Tcal))$.

    \underline{Case 3:} Suppose $k > 1$ and $\ord(\ell) > 3$. 

    \underline{Case 3.1:} Suppose $2 \mid k$. By definition, there exists $m \in \Z$, $0 < m < k$ such that $k=2m$. By Proposition \ref{prop:stackedDistanceFunction}, $d_\Tcal(2m, 0) = d_\Tcal(k,0) = 2 + \frac{1}{2}k = 2 + m$. Similarly, $d_\Tcal(2m, 2k) = d_\Tcal(k,2k) = 2 + \frac{1}{2}k = 2 + m$. As such, the vertex $k$ is equidistant from $0$ and $2k$. Now consider $d_{\Tcal + \Ecal}(k,0)$ and $d_{\Tcal +\Ecal}(k,2k)$. If there exists a path going from $k$ to $0$ or $2k$ which uses $\Ecal$, this path must first reach $0$ or $2k$ to be able to use $\Ecal$. Thus, rendering the use of $\Ecal$ redundant (in terms of finding a minimal path). We now consider the distance $d_{\Tcal + \Ecal}(k, k + \ord(\ell) - 1)$. There exists paths
    \begin{enumerate}
        \item[] $P_1=(k, k+1, k+3, \ldots, 2k+1, 2k+3, \ldots, k + \ord(\ell), k + \ord(\ell)-1)$
        \item[] $P_2=(k, k+1, k-1, \ldots, 1, 2\ord(\ell)-1, \ldots, k + \ord(\ell), k + \ord(\ell)-1)$
        \item[] $P_3=(k, k+1, k+3, \ldots, 2k+1, 2k, 0, 1, 2\ord(\ell)-1, \ldots, k + \ord(\ell), k + \ord(\ell)-1)$
        \item[] $P_4=(k, k+1, k-1, \ldots, 1, 0, 2k, 2k+1, 2k+3, \ldots, k + \ord(\ell), k + \ord(\ell)-1) $
    \end{enumerate}
    This is an exhaustive list of all paths from $k$ to $k+\ord(\ell)-1$ in $\Tcal + \Ecal$. $P_1$ is the shortest path out of all the paths above. In particular, it is the same minimal path used to compute $d_\Tcal(k, k+\ord(\ell)+1)$. Hence 
    \begin{align*}
        d_{\Tcal + \Ecal}(k, k + \ord(\ell)-1) &= \abs{P_1}\\
        &= d_{\Tcal}(k, k+\ord(\ell)-1)\\
        &= 2 + \frac{1}{2}(\ord(\ell) - 1)\\
        &= \frac{1}{2}(\ord(\ell)+3)\\
        &= \diameter(\Tcal)
    \end{align*}
    Thus, we conclude, by Lemma \ref{lem:nonAttachEdgeLowerDiamGraphTheory}, the diameter of $\Tcal$ does not change when $\Ecal$ is added to the graph.

    \underline{Case 3.2:} Suppose $2 \nmid k$. By definition there exists $m \in \Z$, $0 \leq m \leq k$ such that $2m+1 = k$. The reader can verify that by considering $d_{\Tcal + \Ecal}(2m, 2m + \ord(\ell) + 1)$ we will get the value $\frac{1}{2}(\ord(\ell) + 3)$ for similar reasons as in Case 3.1. Hence, we again conclude that the diameter does not change in this case.

    By exhaustive search, we conclude that $\diameter(\Tcal) = \diameter(\Tcal + \Ecal)$ in all cases.
\end{proof}

Now that we know the diameter will not change if $\Ecal$ is contained in a single stacked component, we can now take a look at what happens when $\Ecal$ is contained in the folded component. 

If $\Ecal$ is contained in the folded component, then there exists a new edge incident to two distinct vertices on the floor of $\Fcal$. In particular, since $\Ecal$ cannot be a loop, it must be the case that $\Fcal$ has at least two floor vertices. In other words, we must assume that $\ord(\ell) >3$.

Furthermore, if $\Ecal \in \Theta(\Fcal)$, then there does not exists a unique path between any two vertices in $\Theta(\Fcal)$. This complicates the computation of $\diameter(\Fcal + \Ecal)$. As such, we give a proof idea here and for more details, the reader may get the full proof at Appendix \ref{appendixA}.

\begin{theorem}\label{thrm:diamOfThetaFcal}
    Let $\ord(\ell) > 3$. Suppose $\Ecal$ is an edge in $\Theta(\Fcal)$. Denote $\Ecal = \set{2m, 2(m+k)}$ for $m, k \in \Z$ such that $1 \leq m \leq \frac{1}{2}(\ord(\ell) - 3)$ and $1 \leq k \leq \frac{1}{2}(\ord(\ell) - 1) - m$. Then the value of $\diameter(\Fcal + \Ecal)$

    \begin{enumerate}
        \item is $\max\sbrac*{\frac{1}{2}(\ord(\ell) - 2k+7), \ceil*{\frac{k}{2}} + \frac{1}{2}(\ord(\ell) - 2(m+k) + 5), m + \ceil*{\frac{k}{2}} + 2 }$ if $k \geq 4$;
        \item is $\frac{1}{2}(\ord(\ell) - 1)$ if $k \in \set{2, 3} \text{ and } 2(m+k) = \ord(\ell) - 1$;
        \item is $\frac{1}{2}(\ord(\ell) +1)$ otherwise.
    \end{enumerate}
\end{theorem}

\begin{proof}
    Consider Figure \ref{fig:generalFoldedWithNewEdge}. It depicts what $\Fcal + \Ecal$ looks like in general. We can see that the subgraph made of vertices $\set{0, 1, 2, \ldots, 2m-2}$ is isomorphic to miniature folded component. Similarly with the subgraph made of vertices $\set{2(m+k)+1, 2(m+k) + 2, \ldots, \ord(\ell)-1}$. This subgraph is isomorphic to a miniature folded component with the ``0" vertex missing. Lastly, note that the remaining vertices form a subgraph which looks like a stacked component with a few vertices missing. These realizations allow us to find the eccentricity functions and diameters for each subgraph, then connect them together (using Lemma \ref{lem:edgeAttachementGraphTheory}) to find the diameter of $\Fcal + \Ecal$. Thus, by a few computations we arrive at the desired result.
\end{proof}

We can see that the diameter of $\Fcal+\Ecal$ changes depending on where the new edge is placed. As such, we cannot provided a diameter value of $\Fcal+\Ecal$ without first knowing the precise location of $\Ecal$. However, we can determine a range of values which $\diameter(\Fcal + \Ecal)$ can take.

By the last case of Theorem \ref{thrm:diamOfThetaFcal} we know that if $k=1$ (i.e. $\Ecal$ connects a floor vertex with the subsequent floor vertex) then $\diameter(\Fcal + \Ecal) = \diameter(\Fcal)$. A new edge of this form is possible so long as $\ord(\ell) > 3$. As such, anytime the folded component can have a new edge added to it by $\Theta$, then $\diameter(\Fcal + \Ecal) \leq \frac{1}{2}(\ord(\ell) + 1)$. On the other hand, through a quick computation (available in Appendix \ref{appendixA}) we arrive at the following corollary:

\begin{corollary}\label{cor:minDiameterFcalEcal}
    Suppose $\ord(\ell) > 3$ and the folded component contains the new edge $\Ecal$. Then,
    $$\min(\diameter(\Fcal + \Ecal))  \begin{cases}
        \frac{1}{2}(\ord(\ell) + 1) & \text{if } \ord(\ell) = 5\\
        \frac{1}{2}(\ord(\ell) - 1) & \text{if } \ord(\ell) = 7, 9\\
        \ceil*{\frac{1}{4}(\ord(\ell) + 9)} & \text{if } \ord(\ell) \geq 11
    \end{cases}$$
\end{corollary}

As such we know that if $\ord(\ell) \geq 11$ and the new edge resides in $\Fcal$, then 
\begin{equation}\label{eq:boundsOnDiamFcalPlusEcal}
\ceil*{\frac{\ord(\ell) + 9}{4}} \leq \diameter(\Fcal + \Ecal) \leq \frac{\ord(\ell) + 1}{2}
\end{equation}

\begin{figure}
    \centering
    \includegraphics[width=1\linewidth]{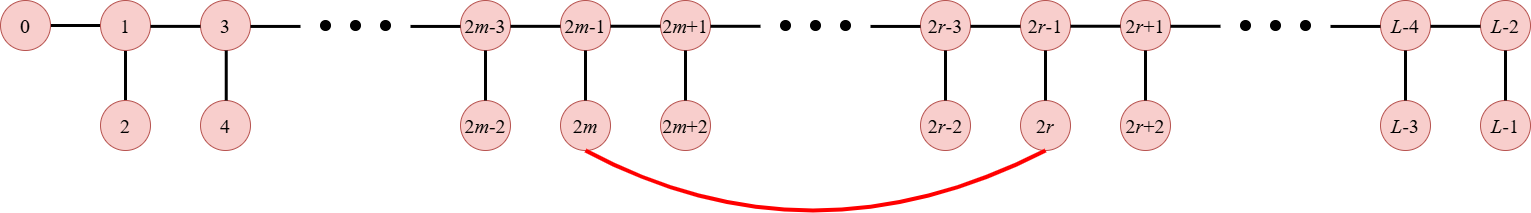}
    \caption{What $\Fcal + \Ecal$ will look like in general where $m,k \in \Z_{>0}$, $r= m+k$, and $L = \ord(\ell)$.}
    \label{fig:generalFoldedWithNewEdge}
\end{figure}

\subsection{$\Ecal$ an Edge Attachment}

As in Subsection \ref{subsec:EcalNotAnAttachingEdge}, we begin this section with a general graph-theoretic lemma.

\begin{lemma}\label{lem:edgeAttachementGraphTheory}
    Let $G$ be a simple graph with distinct components $H$ and $L$. Let $e = \set{v, w}$ be a new edge with $v \in V(H)$ and $w \in V(L)$. Then
    $$\diameter(H + e + L) = \max(\ecc_H(v) + \ecc_L(w) + 1, \diameter(H), \diameter(L))$$
\end{lemma}

\begin{proof}
   Without loss of generality, suppose $G = H + L$ with $H$ and $L$ distinct components. Let $u_1, u_2 \in V(G)$. Note that, if $u_1$, $u_2$ are on the same component then any path in $G+e$ between $u_1$ and $u_2$ must be entirely contained in that same component. This is because, if the path crosses between $H$ and $L$, it must eventually return to the starting component, thereby traversing $e$ twice. This would visit $v$ and $w$ twice, violating the definition of a path. On the other hand, without loss of generality, suppose $u_1 \in V(H)$ and $u_2 \in V(L)$. In such a case there exists a minimal path from $u_1$ to $v$ entirely contained in $H$ and a minimal path from $u_2$ to $w$ entirely contained in $L$. Let $P_H = (u_1, \ldots, v)$ and $P_L = (u_2, \ldots, w)$ be those paths respectively. Now, we can consider the path $P := P_H + e + P_L^{-1} := (u_1, \ldots, v, w, \ldots, u_2)$. Note that, $P$ must be a minimal path since if it was not, then either $P_H$ or $P_L$ would not be minimal. Thus, 
   {\small\begin{align*}
       d(u_1, u_2) &= \abs{P}
       = \abs{P_H + e + P_L^{-1}}
       = \abs{P_H} + 1 + \abs{P_L^{-1}}
       = d_H(u_1, v) + 1 + d_L(w,u_2)
       = d_H(u_1, v) + d_L(u_2, w) + 1
   \end{align*}}
   Hence, we conclude that
   $$d_{G+e}(u_1, u_2) = \begin{cases}
       d_G(u_1, u_2) &\text{if } u_1, u_2 \text{ are on the same component}\\
       d_H(u_1, v) + d_L(u_2, w) + 1 &\text{if } u_1 \in V(H) \text{ and } u_2 \in V(L)
   \end{cases}$$
    Thus, if $u_1 \in V(H)$,
    \begin{align*}
        \ecc_{G+e}(u_1) &= \max\brac*{\max_{u_2 \in V(H)} d_{G + e}(u_1, u_2), \max_{u_2 \in V(L)} d_{G + e}(u_1, u_2)}\\
        &= \max\brac*{\max_{u_2 \in V(H)} d_H(u_1, u_2), d_H(u_1, v) + 1 + \max_{u_2 \in V(L)} d_L(u_2, w)}\\
        &= \max\brac*{\ecc_H(u_1), d_H(u_1, v)+1+\ecc_L(w)}
    \end{align*}
    Similar to the above, if $u_2 \in V(L)$,
    $$\ecc_{G+e}(u_2) = \max\brac*{\ecc_L(u_2), d_L(u_2, w)+1+\ecc_H(v)}$$
    As such we compute
    \begin{align*}
        \diameter(G+e) &= \max_{u \in V(G)} \ecc_{G+e}(u)\\
         &= \max\brac*{\max_{u_1 \in V(H)} \ecc_{G+e}(u_1), \max_{u_2 \in V(L)} \ecc_{G+e}(u_2)}\\
         &= \max\left(\max_{u_1 \in V(H)}\brac*{\max\brac*{\ecc_H(u_1), d_H(u_1, v)+1+\ecc_L(w)}},\right. \\
         &   \hspace{42pt}    \left.\max_{u_2 \in V(L)}\brac*{\max\brac*{\ecc_L(u_2), d_L(u_2, w)+1+\ecc_H(v)}}\right)\\
         &= \max\brac*{\diameter(H), \ecc_H(v) + 1 + \ecc_L(w), \diameter(L), \ecc_L(w) + 1 + \ecc_H(v)}\\
         &= \max(\ecc_H(v) + \ecc_L(w) + 1, \diameter(H), \diameter(L))
    \end{align*}

\end{proof}

The importance of Lemma \ref{lem:edgeAttachementGraphTheory} is immediately apparent for when we are studying edge attachments. However, in our case of looking at when edge attachment occurs in $\Gamma(\glf)$, the equation in Lemma \ref{lem:edgeAttachementGraphTheory} will simplify a great deal. By Proposition \ref{prop:stackedEccFunction} and Theorem \ref{thrm:stackedCompDiameter}, all floor vertices (i.e., even vertices) of a stacked component have eccentricity equal to the component's diameter. Hence they are in the periphery of the stacked component.

By Lemma \ref{lem:newEdgeIsOnFloor} edge attachment occurs only between vertices on the floor of $\Gamma(\glf)$, and since edge attachment must involve a stacked component (due to there only existing one folded component), an attaching edge must involve a vertex in the periphery of one of the components.

Suppose $\Ecal = \set{v_1, v_2}$ is an attaching edge. We know that $v_1$ or $v_2$ is in the periphery of a stacked component. Without loss of generality, assume that $v_2 \in \periphery(\Tcal)$. Let $H$ be the component which $v_1$ resides in. In such a case, by Lemma \ref{lem:edgeAttachementGraphTheory} 
$$\diameter(H + \Ecal + \Tcal) = \max(\ecc_H(v_1) + \ecc_\Tcal(v_2) + 1, \diameter(H), \diameter(\Tcal))$$
From $v_2 \in \periphery(\Tcal)$ we have $\ecc_\Tcal(v_2) = \diameter(\Tcal)$. Thus,
$$\diameter(H + \Ecal + \Tcal) = \max(\ecc_H(v_1) + \diameter(\Tcal) + 1, \diameter(H), \diameter(\Tcal))$$
Since $\ecc_H(v_1) + \diameter(\Tcal) + 1 > \diameter(\Tcal) + 1 >\diameter(\Tcal)$, the expression above simplifies to
\begin{equation}\label{eq:diamComputation}
    \diameter(H + \Ecal + \Tcal) = \max(\ecc_H(v_1) + \diameter(\Tcal) + 1, \diameter(H))
\end{equation}
By definition, $H$ is a component of $\Gamma(\glf)$, hence $H = \Fcal$ or $H = \Tcal$. Note that, if $H = \Tcal$ then by the same reasons as above, we can ignore the $\diameter(H)$ in \eqref{eq:diamComputation}. On the other hand, if $H = \Fcal$ then by Theorems \ref{thrm:foldedCompDiameter} and \ref{thrm:stackedCompDiameter},
$$\diameter(H) = \diameter(\Fcal) = \frac{1}{2}(\ord(\ell) +1) < \frac{1}{2}(\ord(\ell) + 3) = \diameter(\Tcal)< \ecc_H(v_1) + \diameter(\Tcal) + 1$$
Hence we can exclude the $\diameter(H)$ in \eqref{eq:diamComputation}. We conclude that if there is an edge attachment between distinct components $H$ and $\Tcal$ through the edge $\Ecal = \set{v_1,v_2}$ (with $v_1 \in H$ and $v_2 \in \Tcal$), then
\begin{equation}\label{eq:edgeAttachmentChangeInDiameter}
    \diameter(H + \Ecal + \Tcal) = \ecc_H(v_1) + \diameter(\Tcal) + 1
\end{equation}

\begin{remark}\label{rem:edgeAttachmentIncreasesDiameter}
We have shown that the new diameter made by this edge attachment is strictly larger than the diameter of either component involved in the edge attachment.
\end{remark}

Note that, with \eqref{eq:edgeAttachmentChangeInDiameter} we can explicitly determine the diameter of a component made by attaching two stacked components.

\begin{theorem}\label{thrm:diameterOfStackedEdgeAttachment}
    Let $\Tcal_1$ and $\Tcal_2$ be two distinct stacked components of $\Gamma(\glf)$. If $\Ecal$ is an attaching edge between the two components $\Tcal_1$ and $\Tcal_2$, then 
    $$\diameter(\Tcal_1 + \Ecal + \Tcal_2) = \ord(\ell) + 4$$
\end{theorem}

\begin{proof}
    Let $\Ecal= \set{v_1, v_2}$ with $v_1 \in V(\Tcal_1)$ and $v_2 \in V(\Tcal_2)$. Hence by Lemma \ref{lem:newEdgeIsOnFloor}, $v_1 \in \periphery(\Tcal_1)$ making $\ecc_{\Tcal_1}(v_1) = \diameter(\Tcal_1)$. By \eqref{eq:edgeAttachmentChangeInDiameter} and Theorem \ref{thrm:stackedCompDiameter},
    {\small\begin{align*}
        \diameter(\Tcal_1 + \Ecal + \Tcal_2) &= \ecc_{\Tcal_1}(v_1) + \diameter(\Tcal_2) + 1
        = \diameter(\Tcal_1) + \diameter(\Tcal_2) + 1
        = 2 \diameter(\Tcal) + 1
        = \ord(\ell) + 4
    \end{align*}}
\end{proof}

While Theorem \ref{thrm:diameterOfStackedEdgeAttachment} is an excellent and explicit result, we cannot get a similar result for when the folded component is involved in edge attachment. Observe that not every vertex on the floor of $\Fcal$ has the same eccentricity value. As such, we can only get a bound on the value of the new diameter when a folded component is involved in edge attachments.

\begin{theorem}\label{thrm:diameterOfFoldedEdgeAttachment}
    Let $\Fcal$ be the folded component of $\Gamma(\glf)$ and let $\Tcal$ be a stacked component of $\Gamma(\glf)$. If $\Ecal = \set{v_1, v_2}$, with $v_1 \in V(\Fcal)$ and $v_2 \in V(\Tcal)$, is an attaching edge, then 
    $$\frac{3}{4}(\ord(\ell) + 5) \leq \diameter(\Fcal + \Ecal + \Tcal) \leq \ord(\ell) + 3$$
\end{theorem}

\begin{proof}
    Note that, the statement depends on the bounds on $\ecc_{\Fcal}(v_1)$. By Lemma \ref{lem:newEdgeIsOnFloor} $v_1$ must be on the floor, hence $2 | v_1$  and $v_1 \neq 0$ in our labeling of $\Fcal$. Hence, by Proposition \ref{prop:foldedEccEven}
    $$\ecc_\Fcal(v_1) = \begin{cases}
        \frac{1}{2} v_1 + 1 &\text{if } v_1 \geq \frac{1}{2} (\ord(\ell)+1)\\
        \frac{1}{2}(\ord(\ell) - v_1 + 3) &\text{if } v_1 < \frac{1}{2} (\ord(\ell)+1)
    \end{cases}$$
    We already know that $\ecc_\Fcal(v_1) \leq \diameter(\Fcal) = \frac{1}{2}(\ord(\ell) + 1)$ by Theorem \ref{thrm:foldedCompDiameter}. On the other hand, by a very similar computation as the one done in the proof of Theorem \ref{thrm:foldedCompDiameter} we get that $\ecc_\Fcal(v_1) \geq \frac{1}{4}(\ord(\ell) + 5)$. As such 
    $$\frac{1}{4}(\ord(\ell) + 5) \leq \ecc_\Fcal(v_1) \leq \frac{1}{2}(\ord(\ell) + 1)$$
    Which implies
    \begin{align*}
        \diameter(\Fcal + \Ecal + \Tcal) &= \ecc_\Fcal(v_1) + \diameter(\Tcal) + 1
        \leq \frac{1}{2}(\ord(\ell) + 1) + \diameter(\Tcal) + 1
        = \ord(\ell) + 3
    \end{align*}
    and,
    \begin{align*}
        \diameter(\Fcal + \Ecal + \Tcal) &= \ecc_\Fcal(v_1) + \diameter(\Tcal) + 1
        \geq \frac{1}{4}(\ord(\ell) + 5) + \diameter(\Tcal) + 1
        = \frac{3}{4}(\ord(\ell) + 5).
    \end{align*}

\end{proof}

\section{Understanding Mean Diameter}\label{sec:understandingMeanDiameter}

\subsection{Mean Diameter of $\Gamma(\glf)$}

Note that, we now know the following about $\Gamma(\glf)$: 
\begin{enumerate} 
    \item Number of components is equal to $\frac{1}{2}\brac*{\frac{h(-p)}{\ord(\ell)} +1}$ by Corollary \ref{cor:numberOfComponentsOfGammaGlf}, one of which is the folded component (this always exists), and $\frac{1}{2}\brac*{\frac{h(-p)}{\ord(\ell)} -1}$ of which are stacked components; 
    \item The diameter of the folded component is equal to $\frac{1}{2}(\ord(\ell) + 1)$ by Theorem \ref{thrm:foldedCompDiameter}; and 
    \item The diameter of the stacked components (if they exists) is equal to $\frac{1}{2}(\ord(\ell) + 3)$ by Theorem \ref{thrm:stackedCompDiameter}.
\end{enumerate}
As such we can now compute the mean diameter of $\Gamma(\glf)$:
{\footnotesize\begin{align*}
    \MeanDiam\brac*{\Gamma(\glf)} &= \frac{\diameter(\Fcal) + \sum_\Tcal \diameter(\Tcal)}{\frac{1}{2}\brac*{\frac{h(-p)}{\ord(\ell)} +1}}
    = \frac{\diameter(\Fcal) + \brac*{\frac{1}{2}\brac*{\frac{h(-p)}{\ord(\ell)} -1}} \diameter(\Tcal)}{\frac{1}{2}\brac*{\frac{h(-p)}{\ord(\ell)} + 1}}
    = \frac{1}{2}\brac*{\ord(\ell) + 3 - \frac{4\ord(\ell)}{h(-p) + \ord(\ell)}}
\end{align*}}

This information will enable us to determine when an edge attachment happens. This is because the mean diameter of the $\Gamma(\glf)$ will 
\begin{enumerate}
    \item stay the same or decrease when edge attachment does not take place, and 
    \item strictly increase when edge attachment does occur (which will be explicitly proven later).
\end{enumerate}

\subsection{Mean Diameter of $\slp$ without Edge Attachment}\label{subsec:meanDiameterOfSlpWithoutEdgeAttachment}

Note that, if the new edge is not an attaching edge then the change in diameter depends on whether or not the new edge is contained in a stacked component or in the folded component. As such we will look at each case differently. What is true about this entire section is that the number of connected components does not change between $\Gamma(\glf)$ and $\slp$ since $\Ecal$ is not an edge attachment.

\subsubsection{$\Ecal$ In a Stacked Component}\label{subsubsec:EcalInStackedComponent}

For $\Ecal$ to be contained in a stacked component, at least one such component must exist. As such, suppose there exists at least one stacked component. By Corollary \ref{cor:numberOfComponentsOfGammaGlf}, this is equivalent to assuming
$$\frac{1}{2}\brac*{\frac{h(-p)}{\ord(\ell)} - 1} \geq 1 \Leftrightarrow \frac{h(-p)}{\ord(\ell)} \geq 3 \Leftrightarrow h(-p) \geq 3 \ord(\ell)$$
Suppose $\Ecal$ is contained a stacked component $\Tcal_1$. By Theorem \ref{thrm:diameterOfStackedComponentWithNonAttachingEdge}, $\diameter(\Tcal_1) = \diameter(\Tcal_1 + \Ecal)$. Since there is only one new edge, there exist $\frac{1}{2}(\frac{h(-p)}{\ord(\ell)}-1) - 1 = \frac{1}{2}(\frac{h(-p)}{\ord(\ell)}-3)$ stacked components which are not affected at all. As such,
\begin{align*}
    \MeanDiam(\slp) &= \frac{\diameter(\Fcal) + \diameter(\Tcal_1 + \Ecal) + \frac{1}{2}\brac*{\frac{h(-p)}{\ord(\ell)}-3} \diameter(\Tcal)}{\frac{1}{2}\brac*{\frac{h(-p)}{\ord(\ell)} +1}}\\
    &= \frac{\diameter(\Fcal) + \diameter(\Tcal) + \frac{1}{2}\brac*{\frac{h(-p)}{\ord(\ell)}-3} \diameter(\Tcal)}{\frac{1}{2}\brac*{\frac{h(-p)}{\ord(\ell)} +1}}\\
    &= \frac{\diameter(\Fcal) + \frac{1}{2}\brac*{\frac{h(-p)}{\ord(\ell)}-1} \diameter(\Tcal)}{\frac{1}{2}\brac*{\frac{h(-p)}{\ord(\ell)} +1}}\\
    &= \MeanDiam\brac*{\Gamma(\glf)}\\
    &= \frac{1}{2}\brac*{\ord(\ell) + 3 - \frac{4\ord(\ell)}{h(-p) + \ord(\ell)}}
\end{align*}

This is as expected. If the local diameters are not changing and the number of components is not changing, then there is no reason that the mean diameter should change. 

\subsubsection{$\Ecal$ In The Folded Component}

Recall the bounds we provided in \eqref{eq:boundsOnDiamFcalPlusEcal}. On the maximum, we have $\diameter(\Fcal + \Ecal) \leq \diameter(\Fcal)$, hence making $\max(\MeanDiam(\slp))=\MeanDiam(\Gamma(\glf))$. On the other hand, if $\ord(\ell) \geq 11$ we have
{\footnotesize\begin{align*}
    \MeanDiam(\slp) &\geq \frac{\min\brac*{\diameter(\Fcal + \Ecal)} + \frac{1}{2}\brac*{\frac{h(-p)}{\ord(\ell)}-1} \diameter(\Tcal)}{\frac{1}{2}\brac*{\frac{h(-p)}{\ord(\ell)} +1}}
    = \frac{4 \ord(\ell) \ceil*{\frac{1}{4}(\ord(\ell)+9)} + (h(-p) - \ord(\ell))(\ord(\ell) + 3)}{2(h(-p) + \ord(\ell))}
\end{align*}}

Whereas for $\ord(\ell) < 11$, we simply add $\frac{1}{\frac{1}{2}\brac*{\frac{h(-p)}{\ord(\ell)} +1}} = \frac{2\ord(\ell)}{h(-p) + \ord(\ell)}$ to the above. 

\subsection{Mean Diameter of $\slp$ with Edge Attachment}\label{subsec:meanDiameterOfSlpWithEdgeAttachment}

Suppose for this section that $\Ecal$ is an attaching edge. In such a case there must exist at least one stacked component. We will be able to find an explicit value for the mean diameter of $\slp$ when $\Ecal$ is connecting two stacked components. But when $\Ecal$ is connecting the folded component to a stacked component, then we can only provide a range of possible values for the mean diameter of $\slp$. 

In either case, two distinct components merge into one. Hence the total number of components becomes $\frac{1}{2}\brac*{\frac{h(-p)}{\ord(\ell)} + 1} - 1 = \frac{1}{2}\brac*{\frac{h(-p)}{\ord(\ell)} - 1}$.

\subsubsection{$\Ecal$ Connects Two Stacked Components}

Suppose there exist at least two stacked components in $\Gamma(\glf)$. By Corollary \ref{cor:numberOfComponentsOfGammaGlf}, this is equivalent to assuming
$$\frac{1}{2}\brac*{\frac{h(-p)}{\ord(\ell)} - 1} \geq 2 \Leftrightarrow \frac{h(-p)}{\ord(\ell)} \geq 5 \Leftrightarrow h(-p) \geq 5 \ord(\ell)$$
In addition, suppose $\Ecal$ is connecting two distinct stacked components $\Tcal_1$ and $\Tcal_2$. Note that, in such a case we have taken two components and turned them into the component $\Tcal_1 + \Ecal + \Tcal_2$. As such, there exists $\frac{1}{2}\brac*{\frac{h(-p)}{\ord(\ell)} - 1} - 2 = \frac{1}{2}\brac*{\frac{h(-p)}{\ord(\ell)} - 5}$ stacked components which are unaffected by $\Theta$. Lastly, recall that by Theorem \ref{thrm:diameterOfStackedEdgeAttachment}, $\diameter(\Tcal_1 + \Ecal + \Tcal_2) = \ord(\ell)+4$. 

This yields the following computation:
\begin{align*}
    \MeanDiam(\slp) &= \frac{\diameter(\Fcal) + \diameter(\Tcal_1 + \Ecal + \Tcal_2) + \frac{1}{2}\brac*{\frac{h(-p)}{\ord(\ell)} - 5}\diameter(\Tcal)}{\frac{1}{2}\brac*{\frac{h(-p)}{\ord(\ell)} - 1}}
    = \frac{(\ord(\ell) + 3)(h(-p)+\ord(\ell))}{2(h(-p) - \ord(\ell))}
\end{align*}

\subsubsection{$\Ecal$ Connects With The Folded Component}

Suppose that there exists at least one stacked component. By the same reasons as the ones presented at the beginning of Section \ref{subsubsec:EcalInStackedComponent}, this assumption is equivalent to assuming $h(-p) \geq 3 \ord(\ell)$. Further suppose that $\Ecal$ is connecting the folded component and the stacked component $\Tcal_1$. Note that, in such a case there exists $\frac{1}{2}\brac*{\frac{h(-p)}{\ord(\ell)} - 1} - 1 = \frac{1}{2}\brac*{\frac{h(-p)}{\ord(\ell)} - 3}$ stacked components which are unaffected by $\Theta$.

By Theorem \ref{thrm:diameterOfFoldedEdgeAttachment} we know that
$$\frac{3}{4}(\ord(\ell) + 5) \leq \diameter(\Fcal + \Ecal + \Tcal) \leq \ord(\ell) + 3$$
In such a case we get the following upper bound:
\begin{align*}
    \MeanDiam(\slp) &\leq \frac{\max(\diameter(\Fcal + \Ecal + \Tcal_1)) + \frac{1}{2}\brac*{\frac{h(-p)}{\ord(\ell)} - 3}\diameter(\Tcal)}{\frac{1}{2}\brac*{\frac{h(-p)}{\ord(\ell)} - 1}}
    = \frac{(\ord(\ell) + 3)(h(-p)+\ord(\ell))}{2(h(-p) - \ord(\ell))}
\end{align*}
On the other hand, the lower bound of the mean diameter of $\slp$ can be computed to be
\begin{align*}
    \MeanDiam(\slp) &\geq \frac{\min(\diameter(\Fcal + \Ecal + \Tcal_1)) + \frac{1}{2}\brac*{\frac{h(-p)}{\ord(\ell)} - 3}\diameter(\Tcal)}{\frac{1}{2}\brac*{\frac{h(-p)}{\ord(\ell)} - 1}}
    = \frac{h(-p)\ord(\ell) + 3h(-p) + 6\ord(\ell)}{2(h(-p) - \ord(\ell))}
\end{align*}

\subsection{Consequences of Theoretical Mean Diameter}

As discussed at the end of Section \ref{sec:understandingThetaGammaGlf}, we hypothesized that the mean diameter could distinguish between edge attachment and non-edge-attachment. The calculations in Sections \ref{subsec:meanDiameterOfSlpWithoutEdgeAttachment} and \ref{subsec:meanDiameterOfSlpWithEdgeAttachment} confirm this hypothesis.

By the computations in Section \ref{subsec:meanDiameterOfSlpWithoutEdgeAttachment} if $\Theta$ does not induce an edge attachment then the largest value possible for $\MeanDiam(\slp)$ is $\MeanDiam(\Gamma(\glf)) = \frac{1}{2}\brac*{\ord(\ell) + 3 - \frac{4\ord(\ell)}{h(-p) + \ord(\ell)}}$. On the other hand, By the computations in Section \ref{subsec:meanDiameterOfSlpWithEdgeAttachment} if $\Theta$ induces an edge attachment then the lowest value possible for $\MeanDiam(\slp)$ is $\frac{h(-p)\ord(\ell) + 3h(-p) + 6\ord(\ell)}{2(h(-p) - \ord(\ell))}$. 

We know that $h(-p) \geq \ord(\ell)$, but if $h(-p) = \ord(\ell)$ then edge attachment is not possible (since there is only one component). Thus, in order to be able to compare the two cases, we must assume that $h(-p) > \ord(\ell)$. In such a case 
\begin{equation}\label{eq:minMaxEqualityDiameter}
    \frac{h(-p)\ord(\ell) + 3h(-p) + 6\ord(\ell)}{2(h(-p) - \ord(\ell))} = \frac{1}{2}\brac*{\ord(\ell) + 3 - \frac{4\ord(\ell)}{h(-p) + \ord(\ell)}}
\end{equation}
is true if and only if 
$$h(-p)\ord(\ell)(\ord(\ell)+13) = -\ord(\ell)^2(\ord(\ell)+5)$$
through simple algebraic manipulations. Note that, the right side of the equality is always negative and the left side is always positive. Hence, the two values in \eqref{eq:minMaxEqualityDiameter} never intersect. Thus, by looking at a particular valid set of values for $h(-p)$ and $\ord(\ell)$ we can conclude that 

\begin{equation*}
    \frac{h(-p)\ord(\ell) + 3h(-p) + 6\ord(\ell)}{2(h(-p) - \ord(\ell))} > \frac{1}{2}\brac*{\ord(\ell) + 3 - \frac{4\ord(\ell)}{h(-p) + \ord(\ell)}}
\end{equation*}

As such, we conclude that we are able to determine when an edge attachment happens and when it does not happen by the mean diameter value of the spine.

\begin{algorithm}[H]\label{alg:1}
    \caption{Determine Structure of Spine}
    \KwIn{$\ord(\ell)$, $h(-p)$, $\MeanDiam(\slp)$}
    \KwOut{Description of potential structures of $\slp$.}
    \If{$\ord(\ell) < h(-p)$}{
    \If{$\MeanDiam(\slp) = \frac{(\ord(\ell) + 3)(h(-p)+\ord(\ell))}{2(h(-p) - \ord(\ell))}$}{
    \Return{$\slp$ has an edge attachment either between two stacked components or between the folded component and a stacked component. The latter case is true only if the attaching edge is incident to a vertex which is adjacent to $j$-invariant 287496 or $j$-invariant 8000.}
    }
    \ElseIf{$\MeanDiam(\slp) \geq \frac{h(-p)\ord(\ell) + 3h(-p) + 6\ord(\ell)}{2(h(-p) - \ord(\ell))}$}{
    \Return{$\slp$ has an edge attachment between the folded component and a stacked component. In particular the attaching edge is not incident to a vertex which is adjacent to $j$-invariant 287496 or $j$-invariant 8000.}
    }
    \ElseIf{$\MeanDiam(\slp) = \frac{1}{2}\brac*{\ord(\ell) + 3 - \frac{4\ord(\ell)}{h(-p) + \ord(\ell)}}$}{
    \Return{$\slp$ did not have an edge attachment. The new edge is either inside a stacked component or in the folded component. The latter case can only be true if the new edge is of the form $\set{2m, 2(m+k)}$ with either $k=1$, $k = 2$ and $m \neq \frac{\ord(\ell) - 5}{2}$, or $k = 3$ and $m \neq \frac{\ord(\ell) - 7}{2}$.}
    }
    \ElseIf{$\MeanDiam(\slp) < \frac{1}{2}\brac*{\ord(\ell) + 3 - \frac{4\ord(\ell)}{h(-p) + \ord(\ell)}}$}{
    \Return{$\slp$ did not have an edge attachment. In particular, the new edge is contained within the folded component and it is not of the form $\set{2m, 2(m+k)}$ with either $k=1$, $k = 2$ and $m \neq \frac{\ord(\ell) - 5}{2}$, or $k = 3$ and $m \neq \frac{\ord(\ell) - 7}{2}$.}
    }
    }
    \ElseIf{$\ord(\ell) = h(-p)$}{
    \If{$\MeanDiam(\slp) = \frac{1}{2}(\ord(\ell) + 1)$}{
    \Return{There are no stacked components and the new edge is of the form $\set{2m, 2(m+k)}$ with either $k=1$, $k = 2$ and $m \neq \frac{\ord(\ell) - 5}{2}$, or $k = 3$ and $m \neq \frac{\ord(\ell) - 7}{2}$}
    }
    \ElseIf{$\MeanDiam(\slp) < \frac{1}{2}(\ord(\ell) + 1)$}{
    \Return{There are no stacked components and the new edge is not of the form $\set{2m, 2(m+k)}$ with either $k=1$, $k = 2$ and $m \neq \frac{\ord(\ell) - 5}{2}$, or $k = 3$ and $m \neq \frac{\ord(\ell) - 7}{2}$}
    }
    }
\end{algorithm}

Another consequence of the computations in Sections \ref{subsec:meanDiameterOfSlpWithoutEdgeAttachment} and \ref{subsec:meanDiameterOfSlpWithEdgeAttachment} is that there exists cases which we cannot determine which components are incident with the new edge. 

Consider the case when there exists an edge attachment between the folded component and a stacked component which achieves the maximum mean diameter value. At the same time consider the case when there exists an edge attachment between two stacked components. In both cases the value of $\MeanDiam(\slp)$ is $\frac{(\ord(\ell) + 3)(h(-p)+\ord(\ell))}{2(h(-p) - \ord(\ell))}$. In other words, if we know the values of $\ord(\ell)$ and $h(-p)$, and if $\MeanDiam(\slp) = \frac{(\ord(\ell) + 3)(h(-p)+\ord(\ell))}{2(h(-p) - \ord(\ell))}$ then we know that an edge attachment occurred; however, we are not able to determine which of the two cases is occurring. 

A similar issue arises when looking at when $\Ecal$ is a non-attaching-edge. If we know the value of $\ord(\ell)$ and $h(-p)$, and if $\MeanDiam(\slp) = \frac{1}{2}\brac*{\ord(\ell) + 3 - \frac{4\ord(\ell)}{h(-p) + \ord(\ell)}}$, then we know that edge attachment did not occur. However, we cannot determine whether or not the new edge is in a stacked component or in the folded component.

For the sake of readability and conciseness we have put the correct correlations between knowing $\ord(\ell)$, $h(-p)$, $\MeanDiam(\slp)$ and the structure of $\slp$ in Algorithm \ref{alg:1}.

The $j$-invariants 287496 and 8000 appear in Algorithm \ref{alg:1} because, in our labeling, the floor vertices of the folded component in $\periphery(\Tcal)$ are $2$ and $\ord(\ell)-1$. If we go back to the labeling provided by the $j$-invariant of elliptic curves, $\ord(\ell)-1$ is the floor vertex which is adjacent to $j=8000$, and $2$ is the floor vertex which is adjacent to the same vertex which $j=1728$ is adjacent to. Using the modular polynomial for $\ell = 2$ we compute the vertex adjacent to $j=1728$ as $j=287496$. 

\section{Exploring Diameter Data}\label{sec:exploringDiameterData}

Using the University of Calgary's Mathematics and Statistics Server MSX1, we explicitly computed $\slp$ and a component, $\Ccal$, of $\glf$ to determine $\MeanDiam(\slp)$ and $\diameter(\Ccal)$. The reason we computed $\diameter(\Ccal)$ was because $\diameter(\Ccal)$ has the same value as the diameter of a stacked component of $\Gamma(\glf)$. Hence by knowing $\diameter(\Ccal)$ we are able to find $\diameter(\Tcal) = \frac{1}{2}(\ord(\ell) + 3)$ and we are able to isolate for $\ord(\ell)$. In addition we retrived $h(-p)$ for our particular $p$ values by fetching the data from \cite[\nopp\href{https://www.lmfdb.org/NumberField/QuadraticImaginaryClassGroups}{Quadratic Imaginary Class Groups}]{lmfdb}.

The values of $p$ range from 17471 to 2186519 (inclusive) for $p \equiv 71, 119 \pmod{120}$. These bounds were chosen because the value of $\MeanDiam(\slp)$ was not computed for $p$ values strictly higher than 17389, and 2186519 is the ten-thousandth prime which satisfies $p \equiv 71, 119 \pmod{120}$ and is larger than 17389. The duration of the computation of $\MeanDiam(\slp[2])$ and $\diameter(\Ccal)$ for all ten-thousand primes by the MSX1 was approximately 23 hours. We present the results here. The code used to generate this data in MSX1 is in the GitHub repository \cite{MATH518GitHub} in the file \verb|run_diameter.sage|. The sorted output of this code is in the same repository in the file \verb|diameter_results_sorted.csv|.

By knowing $\diameter(\Ccal)$, $h(-p)$, and $\MeanDiam(\slp)$ we are able to follow Algorithm \ref{alg:1} to determine when edge attachment happens and how often. The explicit computations are in the GitHub repository \cite{MATH518GitHub} in the file \verb|Algorithm1.ipynb|. Following the procedure outlined in Algorithm \ref{alg:1} for the ten-thousand primes $p$, we find the results presented in Tables \ref{tab:ord-less} and \ref{tab:ord-equal}.

\begin{table}[h]
\centering

\begin{minipage}[c]{0.49\textwidth}
\centering
\begin{tabular}{lccc}
\toprule
 & \textbf{Max Diam} & \textbf{Not Max Diam} & \textbf{Total} \\
\midrule
Edge Att.: Yes & 694 & 791 & 1485 \\
Edge Att.: No  & 654 & 68  & 722  \\
\midrule
\textbf{Total} & 1348 & 859 & 2207 \\
\bottomrule
\end{tabular}
\caption{$\ord(\ell) < h(-p)$ (2207 primes)}
\label{tab:ord-less}
\end{minipage}
\hfill
\begin{minipage}[c]{0.49\textwidth}
\centering
\begin{tabular}{lc}
\toprule
 & \textbf{Count} \\
\midrule
Max Diam     & 0    \\
Not Max Diam & 7793 \\
\midrule
\textbf{Total}    & 7793 \\
\bottomrule
\end{tabular}
\caption{$\ord(\ell) = h(-p)$ (7793 primes)}
\label{tab:ord-equal}
\end{minipage}

\end{table}

These results are striking. They reveal several trends:

\begin{enumerate}
    \item\label{point1} In 77.93\% of the cases, the prime above 2 generates the class group. Consequently, the spine consists solely of the folded component for this proportion of primes. 
    \begin{enumerate}
        \item Of these primes, none achieved the maximal possible diameter. While not conclusive, this anomalous behavior suggests that the distribution of the new edge may not be entirely random. 
    \end{enumerate}
    \item 22.07\% of the time the prime above 2 does not generate the class group. Meaning that 22.07\% of the time the spine has stacked components.
    \item Of the 22.07\% of primes which have stacked components, approximately $67.29\% \approx \frac{694+791}{2207}\times100$ had edge attachment occur while approximately $32.71\% \approx \frac{654+68}{2207}\times 100$ did not have edge attachment.
    \item\label{point5} Of the edge attachments that occurred, there is approximately an even split between those involving a non-terminal\footnote{A terminal vertex is a vertex which is adjacent to a vertex with $j$-invariant 8000 or 287496} vertex of the folded component and those involving either two stacked components or a terminal vertex of the folded component.
    \item If there is more than one component, then there is approximately a $29.63\% \approx \frac{654}{2207}\times 100$ chance that the new edge is not attaching and maintains the maximal diameter. On the other hand, if there is more than one component, then there is approximately a $3.08\% \approx\frac{68}{2207}\times 100$ chance that the new edge is not attaching and decreases the diameter of the graph.
\end{enumerate}

The conjectured chance of this occurring comes from the Cohen-Lenstra Heuristics \cite{CohenLenstra1984}. In their computations, they state \cite[\nopp C11]{CohenLenstra1984} that the probability that the class group is generated by ideal above 2 is 75.446\%. This precisely explains the behavior of \ref{point1}. We believe that our calculation of 77.93\% is high due to the fact that the convergence to the conjectured heuristic is very slow, and our computations are not large enough to reach them.

Observation \ref{point5} is somewhat unexpected. Since, if there exists more than one stacked component and if the new edge $\Ecal$ is incident to a vertex in a particular stacked component, $\Tcal_1$, then there are twice as many vertices on the floor of stacked components distinct from $\Tcal_1$ than there are vertices on the floor of the folded component. A potential explanation for this unexpected result lies in the assumption that there exists more than one stacked component. It may be possible that most of the time there is only one stacked component. In which case the only possible edge attachment involves the folded component. 

By checking when $3 \ord(\ell) = h(-p)$, we find the results shown in Tables \ref{tab:3Ord(ell)=h(-p)} and \ref{tab:3Ord(ell)<h(-p)}:

\begin{table}[h]
\centering

\begin{minipage}[c]{0.49\textwidth}
\centering
\resizebox{\linewidth}{!}{
\begin{tabular}{lccc}
\toprule
 & \textbf{Max Diam} & \textbf{Not Max Diam} & \textbf{Total} \\
\midrule
Edge Att.: Yes & 0   & 543 & 543 \\
Edge Att.: No  & 464 & 66  & 530 \\
\midrule
\textbf{Total} & 464 & 609 & 1093 \\
\bottomrule
\end{tabular}
}
\caption{$3\,\ord(\ell) = h(-p)$ (1093 primes)}
\label{tab:3Ord(ell)=h(-p)}
\end{minipage}
\hfill
\begin{minipage}[c]{0.49\textwidth}
\centering
\resizebox{\linewidth}{!}{
\begin{tabular}{lccc}
\toprule
 & \textbf{Max Diam} & \textbf{Not Max Diam} & \textbf{Total} \\
\midrule
Edge Att.: Yes & 694 & 248 & 942 \\
Edge Att.: No  & 170 & 2   & 172 \\
\midrule
\textbf{Total} & 864 & 250 & 1114 \\
\bottomrule
\end{tabular}
}
\caption{$3\,\ord(\ell) < h(-p)$ (1114 primes)}
\label{tab:3Ord(ell)<h(-p)}
\end{minipage}

\end{table}

These results are implying the following:
\begin{enumerate}
    \item If there exists more than one component it is equally likely that there is exactly two components or that there exists strictly more than two components.
    \item If there only exists two components, it is very unlikely that edge attachment with the end vertices of the folded component will occur.
    \item If there only exists two components, there is approximately a 50-50 split between the $p$ values which result in an edge attachment and the ones that do not. 
    \item If there are more than two components, it is significantly more probable that the new edge is an attaching edge than not.
    \item If there are more than two components, it is exceedingly rare that a non-attaching edge decreases the mean diameter of the spine.
\end{enumerate}

These conclusions are empirical, based solely on the data gathered, and should not be interpreted as theorems. Furthermore, the sample size is cryptographically small. As such, we cannot definitively claim anything about the structure of the spine. Therefore, definitive claims about the spine's structure cannot be made. Nonetheless, the data indicates a strong tendency for the class group to be generated by the prime above 2. Furthermore, we are able to gather that when the prime above 2 does not generate the class group, edge attachment mostly behaves randomly, but more explicit probabilistic exploration of this is required. Lastly, when there are more than two components in $\Gamma(\glf)$, the probability of the new edge being entirely within the folded component and decreasing the diameter of the folded component is very low. 

The most irregular feature of the data is the folded component's low propensity to receive the new edge as a non-attaching edge, particularly a short one ($k \in \set{1, 2, 3}$). This should be explored more in the future to determine whether or not this behavior is expected or not.

\section{Conclusion}

In this thesis we built on top of the work done by \cite{HedayatArpinScheidler} to explore the structure of the spine of the supersingular $\ell$-isogeny graph, particularly for the case when $\ell = 2$ and $p \equiv 71, 119, \pmod{120}$. We found the diameter of each component in $\Gamma(\glf)$. We used this information to determine the mean diameter of the components of $\slp$ depending on the location of the edge induced by $\Theta$. Finding that the mean diameter of $\slp$ can determine when edge attachment occurs, we provided heuristics on the tendency of edge-attachment. In addition, we found that the mean diameter provides a subset of vertices for which $\Theta$ makes adjacent. A particular case of interest that presented its self was the tendency for the new edge to be relatively long.  

Within future work, the methodology developed in this thesis may be applied to other cases from \cite{HedayatArpinScheidler} where the spine's structure is not fully determined, particularly for $\ell > 2$. In addition, further work needs to be done to understand why there is a lack of small-length edges being added to the spine. Lastly, the bounds on the diameters of the components of the spine may have interesting number-theoretic implications. From the fact of knowing cycles in an isogeny graph result in endomorphisms of elliptic curves, the conjectures give a bound for the degree of such isogenies.

\appendix

\section{Appendix}\label{appendixA}

We state many subsequent lemmas which the reader should be able to verify instantly (or nearly thereof) from the work that was done in Section \ref{sec:UnderstandingGammaGlf}. 

\begin{definition}
    For any $n \in \Z_{\geq 3}$ we define the graph $\bol{\Tcal_n}$ such that $V(\Tcal_n) = \set{0, 1, \ldots, 2n-1}$ and
    $$E(\Tcal_n) = \set{\set{2i, 2i+1}}_{i=0}^{n-1} \cup \set{\set{2i+1, 2i+3}}_{i=0}^{n-2} \cup \set{\set{2n-1, 1}}.$$
\end{definition}
$\Tcal_n$ is a generalization of a stacked component of the spine. In particular, the stacked component of the spine is represented by $\Tcal_{\ord(\ell)}$. 
\begin{lemma}
    \begin{enumerate}
        \item For $u,v \in V(\Tcal_n)$ such that both $u$ and $v$ are odd, 
        $$d_{\Tcal_n}(u,v) = \min\brac*{\frac{1}{2} \abs{u-v}, n - \frac{1}{2} \abs{u, v}}.$$
        \item For $v \in V(\Tcal_n)$,
        $$\ecc_{\Tcal_n}(v) = \begin{cases}
            \floor*{\frac{n}{2}} + 2 &\text{if } 2\mid v\\
            \floor*{\frac{n}{2}} + 1 &\text{if } 2\nmid v
        \end{cases}$$
        \item For any $n$, $\diameter(\Tcal_n) = \floor*{\frac{n}{2}} + 2$
    \end{enumerate}
\end{lemma}

\begin{definition}
    For $n \geq 4$ we define $\bol{\Tcal_n'} := \Tcal_n \setminus \set{0, 2, 2n-2, 2n-4}$.
\end{definition}

\begin{lemma}\label{lem:distance_function_Tcal_prime}
    For all $u,v  \in V(\Tcal_n')$, $d_{\Tcal_n'}(u,v) = d_{\Tcal_n}(u,v)$
\end{lemma}

\begin{proof}
    Just as when we were proving the distance function for $\Tcal$, since no vertex on the floor is visited unless it is either $u$ or $v$, the minimal path between $u$ and $v$ in $\Tcal_n'$ is the same minimal path between $u$ and $v$ in $\Tcal_n$.
\end{proof}

\begin{lemma}\label{lem:ecc_function_Tcal_prime}
    Suppose $n \geq 4$. Denote
    \begin{align*}
        A &= \set{1, 3, 2n-1, 2n-3} \subseteq V(\Tcal_n')\\
        B &= \set{v \in V(\Tcal_n') : 2 \nmid v, \text{ and } 2\ceil*{\frac{n}{2}} - 4 \leq v \leq 2 \floor*{\frac{n}{2}} + 3}\\
        C &= A \setminus (B \cap A)\\
        D &= \set{v \in V(\Tcal_n') : 2 \mid v, \text{ and } 2\ceil*{\frac{n}{2}} - 4 \leq v \leq 2 \floor*{\frac{n}{2}} + 3}
    \end{align*}
    Then for any $v \in \Tcal_n'$,
    $$\ecc(v) = \begin{cases}
        \floor*{\frac{n}{2}}   &\text{if } v \in B\\
        \floor*{\frac{n}{2}}+1 &\text{if } (v \in C \cup D) \text{ or } (2 \nmid v \text{ and } v \notin C \cup B)\\
        \floor*{\frac{n}{2}}+2 &\text{if } 2\mid v \text{ and } v \notin D
    \end{cases}$$
\end{lemma}

\begin{proof}
    The eccentricity of a vertex comes from its distance between its self and a floor vertex incident to a vertex half way across the cycle. As such, the computation follows a simple pattern as to the computation of the eccentricity of the stacked component. The only difference is that the floor vertices $0, 2, 2n-2,$ and $2n-4$ are missing. Hence the floor and surface vertices across the cycle of $\Tcal_n$ have their eccentricities reduced by 1. 
\end{proof}

\begin{lemma}
    For any $n \geq 4$, 
    $$\diameter(\Tcal_n')=\begin{cases}
        \floor*{\frac{n}{2}}     &\text{if } n = 4\\
        \floor*{\frac{n}{2}} + 1 &\text{if } n = 5, 6, 7, 8\\
        \floor*{\frac{n}{2}} + 2 &\text{if } n > 8
    \end{cases}$$
\end{lemma}

\begin{proof}
    The statement comes directly from the eccentricity computation. All we have to do is to determine when each of the cases of the eccentricity function are possible. 
\end{proof}

\begin{definition}
    Let $n$ be an odd integer greater than 1. Define $\bol{\Fcal_n}$ as the graph such that $V(\Fcal_n) = \set{0, 1, \ldots, n-1}$ and 
    $$E(\Fcal_n) = \set{\set{2i, 2i-1}}_{i=1}^{\frac{1}{2}(n-1)} \cup \set{\set{2i-1, 2i+1}}_{i=1}^{\frac{1}{2}(n-3)} \cup \set{\set{0,1}}$$
\end{definition}
Note that, in a way, we have redefined a folded component. In other words, the folded component $\Fcal$ is represented by $\Fcal_{\ord(\ell)}$.

Since this is just a different notation for the folded component, the distance, eccentricity, and diameter of $\Fcal_n$ is the same as for the folded component for when $n = \ord(\ell)$.

\begin{definition}
    For an odd integer greater than 1, $n$, define $\bol{\Fcal_n'} := \Fcal_n \setminus \set{0}$. We define $\bol{\Fcal_1'}$ to be the empty graph.
\end{definition}

\begin{lemma} \label{lem:ecc_dist_functions_Fcal_prime}
    Let $n$ be an odd integer greater than 1. Then
    \begin{enumerate}
        \item For any $u,v \in \Fcal_n'$, $d_{\Fcal_n'}(u,v) = d_{\Fcal_n}(u,v)$.
        \item If $n > 3$, then for any $v \in V(\Fcal_n')$, $\ecc_{\Fcal_n'}(v) = \ecc_{\Fcal_n}(v)$. Otherwise, $\ecc_{\Fcal_n'}(0) = \ecc_{\Fcal_n'}(1) = 1$.
        \item If $n > 3$, then $\diameter(\Fcal_n') = \diameter(\Fcal_n)$. Otherwise $\diameter(\Fcal_n') = 1$.
    \end{enumerate}
\end{lemma}

With all of these lemmas, we now have the ability to prove the main theorem for the diameter of the folded component when the new edge is contained inside of it (i.e. Theorem \ref{thrm:diamOfThetaFcal}).

\begin{proof}
    This proof is a graph theoretic proof and not reliant on any theory of isogeny graphs. Therefore in this proof we will denote $\ord(\ell)$ by $L$. 
    
    \underline{Case 1:} Suppose $k = 1$.
    
    \underline{Case 1.1:} Suppose $2(m+k) \neq L-1$. In such a case, the longest minimal path in $\Fcal_L$ is the path between $0$ and $L-1$. This path has the sub-path $(2m - 1, 2(m+k)-1)$. If we were to redirect this path to use the new edge, we would have the sub-path $(2m-1, 2m, 2(m+k), 2(m+k)-1)$ which is a longer path than we had previously. Hence the diameter does not change when $\Ecal$ is added to $\Fcal_L$.
    
    \underline{Case 1.2:} Suppose $2(m+k) = L-1$. Considering the same minimal path between 0 and $L-1$ we have the sub-path $(L-4, L-2, L-1)$. However, redirecting the path to use the new edge, gives the new sub-path $(L-4, L-3, L-1)$, maintaining the length of the path. Thus, in this case, the diameter does not change either. 
    
    \underline{Case 2:} Suppose $k \in \set{2,3}$. Similar to Case 1.1, redirecting the minimal path of maximal length in $\Fcal_L$ to use the new edge either increases or maintains the length of the path; unless, the new edge is incident to $L-1$, in which case the diameter of $\Fcal_L$ decreases by 1 when the new edge is added (for the $k =3$ case consider the path between 0 and $L-3$).
    
    \underline{Case 3:} Suppose $k \geq 4$. This assumption implies that $L \geq 11$ since lower values of $L$ could not allow $k$ to have a value of 4 or larger. 

    For this next section we will use colours to distinguish between different graphs and their vertices. The colours are there to aid the readers, however, one can use Figure \ref{fig:relabelingFcal} to identify which vertices we are refering to.

    \textbf{Claim:} Up to edge cases, the graph $\Fcal_L + \Ecal$ is isomorphic to the graph
    $$\bol{\Fcal_{2m-1}} + \set{\bol{2m-3}, \rol{3}} + \rol{\Tcal_{k+3}'} + \set{\rol{2k+3}, \tol{1}} + \tol{\Fcal_{L - 2(m+k)}'}$$
    \begin{figure}
        \centering
        \includegraphics[width=1\linewidth]{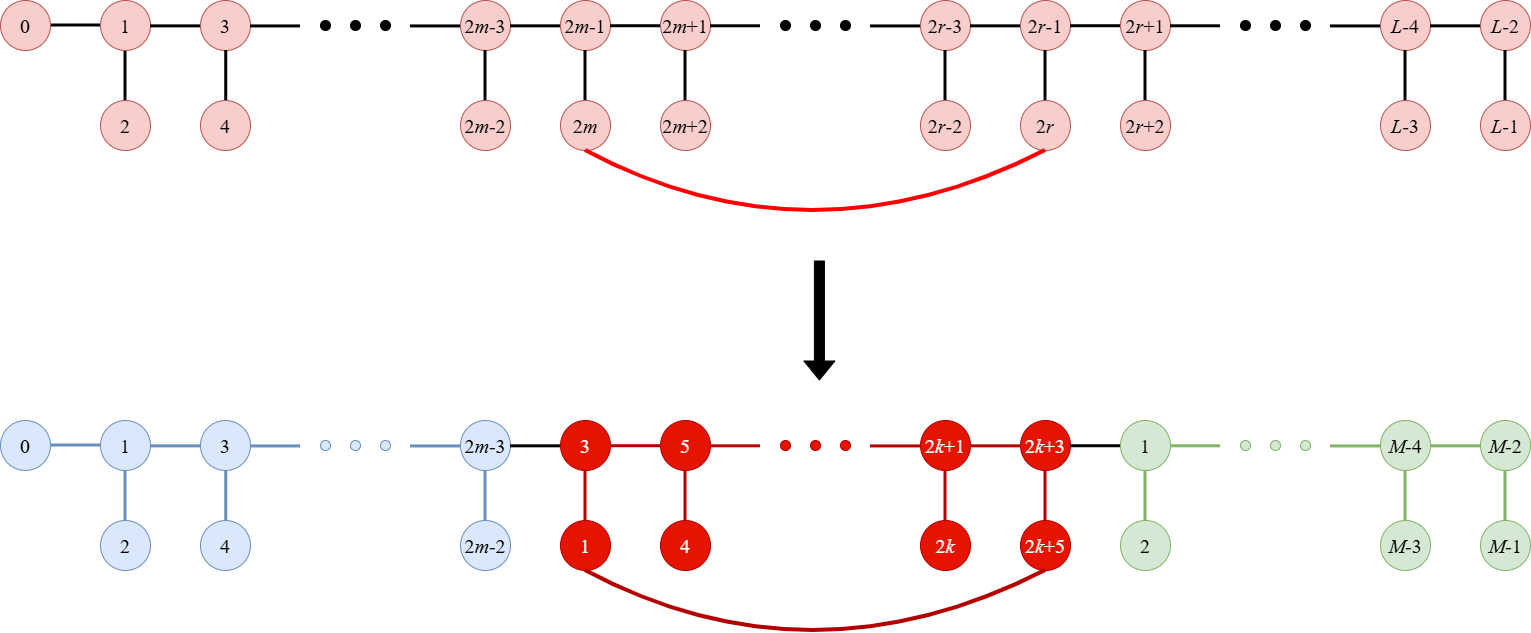}
        \caption{The relabeling of the folded component with the new edge to three distinct components attached to each other. Here, $r = k + m$ and $M = L - 2(k+m)$.}
        \label{fig:relabelingFcal}
    \end{figure}

    We provide Figure \ref{fig:relabelingFcal} as a convincing pictorial proof. Explicitly, we define the function $g$ with domain $\Fcal + \Ecal$ such that 
    $$g(v) = \begin{cases}
        g_1(v) & \text{if } v \leq 2m-2\\
        g_2(v) & \text{if } 2m-1 \leq v \leq 2(m+k)\\
        g_3(v) & \text{if } 2(m+k) + 1 \leq v
    \end{cases}$$
    where $g_1(v) = v$, $g_3(v) = v - 2(m+k)$, and 
    $$g_2(v) = \begin{cases}
        1 &\text{if } v = 2m\\
        2(k+3) - 1 &\text{if } v = 2(m+k)\\
        v - 2m + 4 &\text{if } v \neq 2m, 2(m+k),\text{ and } 2 \nmid v\\
        v - 2m + 2 &\text{if } v \neq 2m, 2(m+k),\text{ and } 2 \mid v
    \end{cases}$$

    We will find the diameter of $\bol{\Fcal_{2m-1}} + \set{\bol{2m-3}, \rol{3}} + \rol{\Tcal_{k+3}'} + \set{\rol{2k+3}, \tol{1}} + \tol{\Fcal_{L - 2(m+k)}'}$ to find the diameter of $\Fcal_L + \Ecal$ since the two are isomorphic. 

    As to avoid edge cases, suppose $k \geq 6$ and $m \geq 2$. In such a case,
    \begin{align*}
        &\diameter\brac*{\bol{\Fcal_{2m-1}} + \set{\bol{2m-3}, \rol{3}} +  \rol{\Tcal_{k+3}'}}\\
        &= \max\brac*{\ecc_{\bol{\Fcal_{2m-1}}}(\bol{2m-3}) + \ecc_{\rol{\Tcal_{k+3}'}}(\rol{3}) + 1, \diameter(\bol{\Fcal_{2m-1}}), \diameter(\rol{\Tcal_{k+3}'}) }\\
        &= \max\brac*{\bol{\frac{1}{2}((2m-1)-1)} + \rol{\floor*{\frac{k+3}{2}}+1} + 1, \bol{\frac{1}{2}((2m-1)+1)},   \rol{\floor*{\frac{k+3}{2}}+2}}\\
        &= \max \brac*{m + \floor*{\frac{k+3}{2}} + 1, m, \floor*{\frac{k+3}{2}} + 2}\\
        &= m + \floor*{\frac{k+3}{2}} + 1\\
        &= m + \ceil*{\frac{k}{2}} + 2
    \end{align*}

    If it were the case that $m=1$, then $\Fcal_{2m-1} = \Fcal_1$ which we notate to be the graph with vertex 0. Due to this being an edge case, most of the formulas break down, However, luckily, the single vertex graph is simple to use. For our purposes, we are interested in the value of 
    \begin{align*} 
    \diameter(\bol{\Fcal_1} + \set{\bol{0}, 3} + \rol{\Tcal_{k+3}'}) &= \max\brac*{\ecc_{\bol{\Fcal_{1}}}(\bol{0}) + \ecc_{\rol{\Tcal_{k+3}'}}(\rol{3}) + 1, \diameter(\bol{\Fcal_{1}}), \diameter(\rol{\Tcal_{k+3}'}) }\\
    &= \max\brac*{\bol{0} + \rol{\floor*{\frac{k+3}{2}}+1} + 1, \bol{0}, \rol{\floor*{\frac{k+3}{2}}+2}}\\
    &= \max\brac*{\floor*{\frac{k+3}{2}}+2, \floor*{\frac{k+3}{2}}+2}\\
    &= \floor*{\frac{k+3}{2}}+2\\
    &= \ceil*{\frac{k}{2}} + 3\\
    &= 1 + \ceil*{\frac{k}{2}} + 2\\
    &= m + \ceil*{\frac{k}{2}} + 2
    \end{align*}

    Checking the cases of ($k=5$ and $m \geq 2$), ($k=5$ and $m =1$), ($k=4$ and $m \geq 2$), and ($k=4$ and $m=2$) by identical computations as above, but substituting the appropriate values, informs us that for all $k \geq 4$ and for all $m$, 
    \begin{equation}\label{eq:diam_Fcal_Tcal}
    \diameter(\bol{\Fcal_{2m-1}} + \set{\bol{2m-3}, \rol{3}} + \rol{\Tcal_{k+3}'}) = m + \ceil*{\frac{k}{2}} + 2
    \end{equation}

    Using Lemma \ref{lem:edgeAttachementGraphTheory} again, to find $\diameter\brac*{\bol{\Fcal_{2m-1}} + \set{\bol{2m-3}, \rol{3}} + \rol{\Tcal_{k+3}'} + \set{\rol{2k+3}, \tol{1}} + \tol{\Fcal_{L - 2(m+k)}'}}$ equal to
    \begin{multline}\label{eq:main_diam_Fcal_Ecal_comp}
        \max \left(\ecc_{\bol{\Fcal_{2m-1}} + \set{\bol{2m-3}, \rol{3}} + \rol{\Tcal_{k+3}'}}(\rol{2k+3}) + \ecc_{\tol{\Fcal_{L - 2(m+k)}'}}(\tol{1}) + 1,\right. \\
            \left. \diameter\brac*{\bol{\Fcal_{2m-1}} + \set{\bol{2m-3}, \rol{3}} + \rol{\Tcal_{k+3}'}}, \diameter\brac*{\tol{\Fcal_{L - 2(m+k)}'}} \right)
    \end{multline}
    The eccentricity of $\rol{2k+3}$ in the graph $\bol{\Fcal_{2m-1}} + \set{\bol{2m-3}, \rol{3}} + \rol{\Tcal_{k+3}'}$ is determined by
    $$\max\brac*{\ecc_{\bol{\Fcal_{2m-1}}}(\bol{2m-3}) + d_{\rol{\Tcal_{k+3}'}}(\rol{2k+3, 3}) + 1, \ecc_{\rol{\Tcal_{k+3}'}}(2k+3)}$$
    
    By Lemmas \ref{lem:distance_function_Tcal_prime} and \ref{lem:ecc_function_Tcal_prime}, if $k \geq 4$, then $\ecc_{\rol{\Tcal_{k+3}'}}(\rol{2k+3}) = \floor*{\frac{k+3}{2}}$ and $d_{\rol{\Tcal_{k+3}'}}(\rol{2k+3, 3}) = 3$. Furthermore, for any $m$, $\ecc_{\bol{\Fcal_{2m-1}}}(\bol{2m-3}) = m-1$. As such, if $k \geq 4$, then
    $$\ecc_{_{\bol{\Fcal_{2m-1}} + \set{\bol{2m-3}, \rol{3}} + \rol{\Tcal_{k+3}'}}}(\rol{2k+3}) = \max\brac*{m + 3, \floor*{\frac{k+3}{2}}+1}$$
    
    Note that, in the case of $m = 1$, the blue component is the single vertex graph. In this case, the attachment between the blue component and the red component happens between the vertices $\bol{0}$ and $\rol{3}$. As such, $\ecc_{\bol{\Fcal_{2m-1}}}(\bol{0}) = 0 = m-1$, hence the formula still holds. 

    Turning our attention to the other end of the graph, by Lemma \ref{lem:ecc_dist_functions_Fcal_prime}, if $k \leq \frac{1}{2}(L - 5) - m$, then $\ecc_{\tol{\Fcal_{L - 2(m+k)}'}}(\tol{1}) = \frac{1}{2}(L - 2(m+k) - 1)$ and $\diameter\brac*{\tol{\Fcal_{L - 2(m+k)}'}} = \frac{1}{2}(L - 2(m+k) + 1)$. Putting all of our information together, if $4 \leq k \leq \frac{1}{2}(L - 5) - m$, then the value of \eqref{eq:main_diam_Fcal_Ecal_comp} is equal to 
    \begin{equation}\label{eq:max_comp_simplified_once}
        \max\brac*{\max\brac*{m+3, \ceil*{\frac{k}{2}} + 2} + \frac{1}{2}(L - 2(m+k) - 1) + 1, m + \ceil*{\frac{k}{2}} + 2, \frac{1}{2}(L - 2(m+k) + 1)}
    \end{equation}
    Realizing that $\frac{1}{2}(L - 2(m+k) - 1) + 1 \geq \diameter\brac*{\tol{\Fcal_{L - 2(m+k)}'}}$, we may get rid of the third value in the maximum of \eqref{eq:max_comp_simplified_once}. In fact, by taking the diameter of the empty graph to be 0, the simplification holds even for when $k =\frac{1}{2}(L - 3) - m , \frac{1}{2}(L - 1) - m$.

    Thus, if $4 \leq k$, the diameter of $\bol{\Fcal_{2m-1}} + \rol{\Tcal_{k+3}'} + \tol{\Fcal_{L - 2(m+k)}'}$ is equal to
    \begin{equation*}
        \max\brac*{\max\brac*{m+3, \ceil*{\frac{k}{2}} + 2} + \frac{1}{2}(L - 2(m+k) + 1), m + \ceil*{\frac{k}{2}} + 2}
    \end{equation*}
    By pulling out the internal maximum, the diameter equals to
    \begin{equation*}
        \max\sbrac*{\frac{1}{2}(\ord(\ell) - 2k+7), \ceil*{\frac{k}{2}} + \frac{1}{2}(\ord(\ell) - 2(m+k) + 5), m + \ceil*{\frac{k}{2}} + 2 }
    \end{equation*}
\end{proof}

With Theorem \ref{thrm:diamOfThetaFcal} proven we may prove Corollary \ref{cor:minDiameterFcalEcal}.

\begin{proof}
    Suppose $L \geq 5$ (otherwise a new edge cannot appear in the folded component). If $L = 5$, there is only one valid location for the new edge which then forces $\diameter(\Fcal + \Ecal) = \frac{1}{2}(L+1)$. Thus, we may suppose $L \geq 7$. 
    
    In this case, there exists valid $m$ and $k$ values such that $k \in \set{2,3}$ and $2(m+k) = L-1$. In other words, there exists valid $m$ and $k$ such that $\diameter(\Fcal + \Ecal) = \frac{1}{2}(L-1) < \frac{1}{2}(L+1)$. Hence, we may exclude the ``otherwise" case of Theorem \ref{thrm:diamOfThetaFcal} in our computation of the minimum diameter. 

    For brevity, denote the large maximum function in Theorem \ref{thrm:diamOfThetaFcal} by $\max(f_1, f_2, f_3)$. Consider, 
    $$f_2 + f_3 = 2 \ceil*{\frac{k}{2}} + \frac{L - 2k + 9}{2} \geq \frac{L+9}{2}.$$
    Using the above, the fact that for any $a,b \in \Z_{>0}, \max(a,b) \geq \frac{a+b}{2}$, and $\max(f_2, f_3) \in \Z$, we get
    $$\max(f_2, f_3) \geq \ceil*{\frac{L+9}{4}}$$
    Furthermore, by some simple manipulations, we have the following equivalent statements
    \begin{equation*}
        f_1 \leq \frac{L+9}{4} \Longleftrightarrow \frac{1}{2}(L-2k+7) \leq \frac{L+9}{4} \Longleftrightarrow 11 \leq L
    \end{equation*}

    Thus, when $L \geq 11$, the minimum diameter value is achieved by $\ceil*{\frac{L+9}{4}}$, and the other cases can be explicitly computed. 
    
\end{proof}

\newpage

\printbibliography

\end{document}